\documentclass[11pt]{article}

\usepackage{amssymb}
\usepackage{amsmath,amsfonts,mathtools}
\usepackage{amsthm}
\usepackage{latexsym}
\usepackage{mathrsfs}
\usepackage{color}
\usepackage{float}
\usepackage{enumitem}
\usepackage{algorithm,algorithmic}
\usepackage{soul}
 \allowdisplaybreaks[3]

\usepackage[a4paper]{geometry}
\newcommand{\beq}{\begin{equation}}
\newcommand{\eeq}{\end{equation}}
\newcommand{\beqa}{\begin{eqnarray}}
\newcommand{\eeqa}{\end{eqnarray}}
\newcommand{\beqas}{\begin{eqnarray*}}
\newcommand{\eeqas}{\end{eqnarray*}}
\newcommand{\ba}{\begin{array}}
\newcommand{\ea}{\end{array}}
\newcommand{\bi}{\begin{itemize}}
\newcommand{\ei}{\end{itemize}}
\newcommand{\nn}{\nonumber}

\DeclareMathOperator*{\Argmax}{Argmax}  

\newcommand{\mcX}{{\mathcal X}}
\newcommand{\mcY}{{\mathcal Y}}
\newcommand{\mcL}{{\mathcal L}}

\newcommand{\mcT}{{\mathcal T}}

\newcommand{\prox}{\mathrm{prox}}
\newcommand{\dom}{\mathrm{dom}}
\newcommand{\dist}{\mathrm{dist}}
\newcommand{\argmin}{\arg\min}
\newcommand{\argmax}{\arg\max}

\newtheorem{lemma}{Lemma}
\newtheorem{thm}{Theorem}

\newtheorem{defi}{Definition}

\newtheorem{assumption}{Assumption}

\newtheorem{rem}{Remark}

\newcounter{spb}
\def\cA{{\cal A}}
\def\cB{{\mathbb B}}

\def\cF{{\cal F}}

\def\cO{{\cal O}}

\def\cA{{\cal A}}

\def\tlambda{{\tilde \lambda}}

\def\wT{{\widehat T}}
\def\tg{{\tilde g}}

\def\tx{{\tilde x}}
\def\ty{{\tilde y}}

\def\bR{{\mathbb{R}}}
\def\halpha{{\hat\alpha}}
\def\hdelta{{\hat\delta}}

\def\tf{{\tilde f}}

\def\xe{{x_\epsilon}}
\def\ye{{y_\epsilon}}

\def\bbK{\mathbb{K}}
\def\bbT{\mathbb{T}}
\def\tl{\tilde\lambda}

\def\bh{{h}}
\def\bH{{H}}
\def\hh{{\bH^*_{\epsilon}}}

\def\h{\bar h}
\def\H{\bar H}
\def\cP{{\cal P}}
\def\cG{{\cal G}}
\def\tp{{\widetilde\partial}}

\def\AL{{\mcL}}

\def\np{{p}}
\def\nq{{q}}

\def\bbf{{\bar f}}
\def\bp{{\bar p}}
\def\tp{{\tilde p}}
\def\bff{{\rm \bf nf}}
\def\bfx{{\rm \bf x}}
\def\bfy{{\rm \bf y}}

\def\tm{{\tilde m}}
\def\tn{{\tilde n}}

\definecolor{ngreen}{RGB}{38,217,169}

\title{A first-order augmented Lagrangian method for constrained minimax optimization}
\author{
Zhaosong Lu
\thanks{
Department of Industrial and Systems Engineering, University of Minnesota, USA (email: {\tt zhaosong@umn.edu}, {\tt mei00035@umn.edu}). This work was partially supported by NSF Award IIS-2211491, ONR Award N00014-24-1-2702, and AFOSR Award FA9550-24-1-0343.}
\and
Sanyou Mei
\footnotemark[2]
}

\date{January 5, 2023 (Revised: October 28, 2024)}

\begin{document}
\maketitle

\begin{abstract}
In this paper we study a class of constrained minimax problems. In particular, we propose a first-order augmented Lagrangian method for solving them, whose subproblems turn out to be a much simpler structured minimax problem and are suitably solved by a first-order method developed in this paper. Under some suitable assumptions, an \emph{operation complexity} of $\cO(\varepsilon^{-4}\log\varepsilon^{-1})$, measured by its fundamental operations, is established for the first-order augmented Lagrangian method for finding an $\varepsilon$-KKT solution of the constrained minimax problems. 
\end{abstract}


\medskip


\section{Introduction}

In this paper, we consider a constrained minimax problem
\beq\label{prob}
F^*=\min_{c(x)\leq 0}\max_{d(x,y)\leq 0}\{F(x,y):=f(x,y)+p(x)-q(y)\}.
\eeq
For notational convenience, throughout this paper we let $\mcX:=\dom\,p$ and $\mcY:=\dom\,q$, where $\dom\,p$ and $\dom\,q$ are the domain of $p$ and $q$, respectively. 
Assume that problem \eqref{prob} has at least one optimal solution and the following additional assumptions hold.
\begin{assumption}\label{a1}
\begin{enumerate}[label=(\roman*)]
\item $f$ is $L_{\nabla f}$-smooth on $\mcX\times\mcY$ and $f(x,\cdot)$ is concave for any given $x\in\mcX$.\footnote{The definitions of $L_{\phi}$-Lipschitz continuity and $L_{\nabla \phi}$-smoothness of a function or mapping $\phi$ are given in Subsection \ref{notation}.}
\item $p:\bR^n\to\bR\cup\{+\infty\}$ and $q:\bR^m\to\bR\cup\{+\infty\}$ are proper closed convex functions, and the proximal operator of $p$ and $q$ can be exactly evaluated.
\item $c:\bR^n\to\bR^{\tn}$ is $L_{\nabla c}$-smooth and $L_c$-Lipschitz continuous on $\mcX$, $d:\bR^n\times\bR^m\to\bR^{\tm}$ is $L_{\nabla d}$-smooth and $L_d$-Lipschitz continuous on $\mcX\times\mcY$, and each component $d_i(x,\cdot)$ of $d$ is convex for all $i=1,\ldots, \tm$ and  $x\in\mcX$.
\item The sets $\mcX$ and $\mcY$ (namely, $\dom\,p$ and $\dom\,q$) are compact.
\end{enumerate}
\end{assumption}

Problem \eqref{prob} has found applications in machine learning such as perceptual adversarial robustness \cite{laidlaw2020perceptual} and robust adversarial classification \cite{ho2023adversarial}.  Besides, it has potential application to constrained bilevel optimization
 \begin{equation}\label{BLO}
\min\limits_{x,y} \bbf(x,y)+\bp(x) \quad  {\rm s.t.}  \quad y\in\argmin\limits_z\{\tf(x,z)+\tp(z)|\tg(x,z)\leq0\},               
\end{equation}
where $\bp$ and $\tp$ are proper closed convex functions, $\tg$, $\nabla \bbf$, $\nabla \tf$  and $\nabla \tg$ are Lipschitz continuous on $\dom\,\bp \times \dom\,\tp$,  and $\tg_i(x,\cdot)$ is convex for each $x\in\dom\,\bp$. Specifically, \eqref{BLO} can be tackled by solving a sequence of subproblems in the form of \eqref{prob}. Indeed, observe that \eqref{BLO} is equivalent to
 \begin{equation}\label{BLO-ref}
\min\limits_{x,y} \bbf(x,y)+\bp(x) \quad  {\rm s.t.}  \quad \tg(x,y)\leq0, \ \  \tf(x,y)+\tp(y)- \min\limits_z\{\tf(x,z)+\tp(z)|\tg(x,z)\leq0\} \leq 0.                
\end{equation}
Notice that any feasible point $(x,y)$ of \eqref{BLO-ref} satisfies $\tf(x,y)+\tp(y)- \min_z\{\tf(x,z)+\tp(z)|\tg(x,z)\leq0\} \geq 0$. As a result, one natural approach to tackling  \eqref{BLO-ref} is by solving a sequence of penalty subproblems in the form of
\[
\min\limits_{\tg(x,y)\leq0} \big\{\bbf(x,y)+\bp(x)+\rho\big(\tf(x,y)+\tp(y)- \min\limits_z\{\tf(x,z)+\tp(z)|\tg(x,z)\leq0\}\big)\big\},               
\]
which turns out to be a special case of  \eqref{prob}  given by 
\[
\min\limits_{\tg(x,y)\leq0} \max\limits_{\tg(x,z)\leq0}\big\{\bbf(x,y)+\rho\big(\tf(x,y)- \tf(x,z)\big)+\bp(x)-\rho\tp(z)\big\}.              
\]

In the recent years, the minimax problem of a simpler form
\beq \label{special-minimax}
\min_{x\in X}\max_{y\in Y} f(x; y),
\eeq
where $X$ and $Y$ are closed sets, has received tremendous amount of attention. Indeed, it has found broad applications in many areas, such as adversarial training \cite{Good15,Madry18,Sin18,Wang21}, generative adversarial networks \cite{Gid19,Good14,sanj18}, reinforcement learning \cite{Dai18,Du17,Na19,qiu20,song18}, computational game \cite{Ant21,Rak13,Syr15}, distributed computing \cite{Mat10,Sha08},  prediction and regression \cite{Ces06,Tas06,Xu09,XuNe05}, and distributionally robust optimization \cite{Duc19,Sha15}. Numerous methods have been developed for solving  \eqref{special-minimax} with $X$ and $Y$ being \emph{simple closed convex sets} (e.g., see \cite{chen2021proximal,guo2023fast,huang2022accelerated,lin20b,lin2020near,lu2020hybrid,
luo2020stochastic,nou19,xian2021faster,xu2020gradient,xu2023unified,zhang2020single}). 

There have also been several studies on some other special cases of problem  \eqref{prob}. In particular, two first-order methods, called max-oracle gradient-descent and nested gradient descent/ascent methods, were proposed in \cite{goktas2021convex} for solving \eqref{prob} with $c(x)\equiv0$ and $p$ and $q$ being respectively the indicator function of simple compact convex sets $X$ and $Y$, under the assumption that $V(x)=\max_{y\in Y} \{f(x,y): d(x,y)\leq0\}$ is convex and moreover an optimal Lagrangian multiplier associated with the constraint $d(x,y)\leq 0$ can be computed for each $x\in X$. An augmented Lagrangian (AL) method was recently proposed in \cite{dai2022rate} for solving \eqref{prob} with {\it only equality constraints}, $p(x)\equiv0$, $q(y)\equiv0$ and $c(x)\equiv0$, under the assumption that a {\it local min-max point} of the AL subproblem can be found at each iteration. In addition,  a multiplier gradient descent method was proposed in \cite{tsaknakis2023minimax} for solving 
\eqref{prob} with $c(x)\equiv0$, $d(x,y)$ being an \emph{affine} mapping, and $p$ and $q$ being the indicator function of simple compact convex sets. Also, a proximal gradient multi-step ascent decent method was developed in \cite{dai2022optimality} 
for \eqref{prob} with $c(x)\equiv0$, $d(x,y)$ being an \emph{affine} mapping and $f(x,y)=g(x)+x^TAy-h(y)$, under the assumption that $f(x,y)-q(y)$ is \emph{strongly concave} in $y$. Besides, primal dual alternating proximal gradient methods were proposed in \cite{zhang2022primal} for \eqref{prob} with $c(x)\equiv0$, $d(x,y)$ being an \emph{affine} mapping, and \{$f(x,y)$ being strongly concave in $y$ or [$q(y)\equiv0$ and $f(x,y)$ being a linear function in $y$]\}. An iteration complexity of the method for finding an approximate stationary point of the aforementioned special minimax problem was established in  \cite{dai2022optimality,goktas2021convex,zhang2022primal}, respectively. Yet, their operation complexity, measured by the number of fundamental operations such as evaluations of gradient  of $f$ and proximal operator of $p$ and $q$, was not studied in these works. 

There was no algorithmic development for \eqref{prob} prior to our work, though optimality conditions of \eqref{prob} were recently studied in \cite{dai2020optimality}. In this paper, we propose a first-order AL method for solving  \eqref{prob}. Specifically, given an iterate $(x^k,y^k)$ and a Lagrangian multiplier estimate $(\lambda^k_\bfx,\lambda^k_\bfy)$ at the $k$th iteration,  the next iterate $(x^{k+1},y^{k+1})$ is obtained by finding an approximate stationary point of the AL subproblem
\[
\min_x\max_y \AL(x,y,\lambda^k_\bfx,\lambda^k_\bfy;\rho_k)
\]
for some $\rho_k>0$ through the use of a first-order method proposed in this paper, 
where $\AL$ is the AL function of \eqref{prob} defined as 
\beq\label{AL}
\AL(x,y,\lambda_\bfx,\lambda_\bfy;\rho)=F(x,y)+\frac{1}{2\rho}\left(\|[\lambda_\bfx+\rho c(x)]_+\|^2-\|\lambda_\bfx\|^2\right)-\frac{1}{2\rho}\left(\|[\lambda_\bfy+\rho d(x,y)]_+\|^2-\|\lambda_\bfy\|^2\right),
\eeq
which is a generalization of the AL function introduced in \cite{dai2022rate} for an equality constrained minimax problem. The Lagrangian multiplier estimate is then updated by $\lambda_\bfx^{k+1}=\Pi_{\cB^+_\Lambda}(\lambda^k_\bfx+\rho_kc(x^{k+1}))$ and $\lambda^{k+1}_\bfy=[\lambda^k_\bfy+\rho_kd(x^{k+1},y^{k+1})]_+$  for some $\Lambda>0$, where $\Pi_{\cB^+_\Lambda}(\cdot)$ and $[\cdot]_+$ are defined in Section \ref{notation}. 

\vspace{.1in}

The main contributions of this paper are summarized below.

\begin{itemize}
\item  We propose a first-order AL method for solving problem \eqref{prob}. To the best of our knowledge, this is the first yet implementable method for solving \eqref{prob}.
\item We show that under some suitable assumptions, our first-order AL method enjoys an iteration complexity of $\cO(\log\varepsilon^{-1})$ and an operation complexity of $\cO(\varepsilon^{-4}\log\varepsilon^{-1})$, measured by the number of evaluations of $\nabla f$, $\nabla c$, $\nabla d$ and proximal operator of $p$ and $q$, for finding an $\varepsilon$-KKT solution of \eqref{prob}.
\end{itemize}

The rest of this paper is organized as follows. In Subsection \ref{notation}, we introduce some notation and terminology. In Section \ref{minimax}, we propose a first-order method for solving a nonconvex-concave minimax problem and study its complexity. In Section~\ref{sec:main}, we propose a first-order AL method for solving problem \eqref{prob} and present complexity results for it. Finally, we provide the proof of the main results in Section \ref{sec:proofs}.

\subsection{Notation and terminology}  \label{notation}
The following notation will be used throughout this paper. Let $\bR^n$ denote the Euclidean space of dimension $n$ and $\bR^n_+$ denote the nonnegative orthant in $\bR^n$. The standard inner product, $l_1$-norm and Euclidean norm are denoted by $\langle\cdot,\cdot\rangle$, $\|\cdot\|_1$ and $\|\cdot\|$, respectively. For any $\Lambda>0$, let $\cB^+_\Lambda =\{x \geq 0:\|x\|\leq \Lambda\}$, whose dimension is clear from the context. 
For any $v\in\bR^n$, let $v_+$ denote the nonnegative part of $v$, that is, $(v_+)_i=\max\{v_i,0\}$ for all $i$. Given a point $x$ and a closed set $S$ in $\bR^n$, let 
$\dist(x,S)=\min_{x'\in S} \|x'-x\|$, $\Pi_{S}(x)$ denote the Euclidean projection of $x$ onto $S$, and $\delta_S$ denote the indicator function associated with $S$.

A function or mapping $\phi$ is said to be \emph{$L_{\phi}$-Lipschitz continuous} on a set $S$ if $\|\phi(x)-\phi(x')\| \leq L_{\phi} \|x-x'\|$ for all $x,x'\in S$. In addition, it is said to be \emph{$L_{\nabla\phi}$-smooth} on $S$ if $\|\nabla\phi(x)-\nabla\phi(x')\| \leq L_{\nabla\phi} \|x-x'\|$ for all $x,x'\in S$. For a closed convex function $p:\bR^n\to \bR\cup\{+\infty\}$, the \emph{proximal operator} associated with $p$ is denoted by  
$\prox_p$,  that is,
\beq\label{eq:def-prox}
\prox_p(x) = \argmin_{x'\in\bR^n} \left\{ \frac{1}{2}\|x' - x\|^2 + p(x') \right\} \quad \forall x \in \bR^n.
\eeq
Given that evaluation of $\prox_{\gamma p}(x)$ is often as cheap as $\prox_p(x)$, we count the evaluation of $\prox_{\gamma p}(x)$ as one evaluation of proximal operator of $p$ for any $\gamma>0$ and $x\in\bR^n$. 

For a lower semicontinuous function $\phi:\bR^n\to \bR\cup\{+\infty\}$, its \emph{domain} is the set $\dom\, \phi := \{x| \phi(x)<+\infty\}$. The \emph{upper subderivative} of $\phi$ at $x\in \dom\, \phi$ in a direction $d\in\bR^n$ is defined by
\[
\phi'(x;d) = \limsup\limits_{x' \stackrel{\phi}{\to} x,\, t \downarrow 0} \inf_{d' \to d} \frac{\phi(x'+td')-\phi(x')}{t},
\] 
where $t\downarrow 0$ means both $t > 0$ and $t\to 0$, and $x' \stackrel{\phi}{\to} x$ means both $x' \to x$ and $\phi(x')\to \phi(x)$. The \emph{subdifferential} of $\phi$ at $x\in \dom\, \phi$ is the set 
\[
\partial \phi(x) = \{s\in\bR^n\big| s^T d \leq \phi'(x; d) \ \ \forall d\in\bR^n\}.
\]
We use $\partial_{x_i} \phi(x)$ to denote the subdifferential with respect to $x_i$.  
In addition, for an upper semicontinuous function $\phi$, its subdifferential is defined as $\partial \phi=-\partial (-\phi)$. If $\phi$ is locally Lipschitz continuous, the above definition of subdifferential coincides with the Clarke subdifferential. Besides, if $\phi$ is convex, it coincides with the ordinary subdifferential for convex functions. Also, if $\phi$ is continuously differentiable at $x$ , we simply have $\partial \phi(x) = \{\nabla \phi(x)\}$, where $\nabla \phi(x)$ is the gradient of $\phi$ at $x$. In addition, it is not hard to verify that $\partial (\phi_1+\phi_2)(x)=\nabla \phi_1(x)+\partial \phi_2(x)$ if $\phi_1$ is continuously differentiable at $x$ and $\phi_2$ is lower or upper semicontinuous at $x$. See \cite{clarke1990optimization,ward1987nonsmooth} for more details.

Finally, we introduce an (approximate) primal-dual stationary point (e.g., see \cite{dai2022optimality,dai2020optimality,kong2021accelerated}) for a general minimax problem
\begin{equation}\label{eg}
\min_{x}\max_{y}\Psi(x,y),
\end{equation}
where $\Psi(\cdot,y): \bR^n \to \bR \cup\{+\infty\}$ is a lower semicontinuous function, and $\Psi(x,\cdot): \bR^m \to \bR \cup\{-\infty\}$ is an upper semicontinuous function.

\begin{defi}  \label{def2}
A point $(x,y)$ is said to be a primal-dual stationary point of the minimax problem \eqref{eg} if 
\[
0 \in \partial_x\Psi(x,y), \quad 0\in\partial_y\Psi(x,y).
\]
In addition, for any $\epsilon>0$, a point $(\xe,\ye)$ is said to be an $\epsilon$-primal-dual stationary point of the minimax problem \eqref{eg} if
\begin{equation*}
\dist\left(0,\partial_x\Psi(\xe,\ye)\right)\leq\epsilon,\quad\dist\left(0,\partial_y\Psi(\xe,\ye)\right)\leq\epsilon.
\end{equation*}
\end{defi}
One can see that $(\xe,\ye)$ is an $\epsilon$-primal-dual stationary point of  \eqref{eg} if and only if 
$\xe$ and $\ye$ are an $\epsilon$-stationary point of $\min_{x}\Psi(x,\ye)$ and $\max_{y}\Psi(\xe,y)$, respectively.

\section{A first-order method for nonconvex-concave minimax problem} 
\label{minimax}

In this section, we propose a first-order method for finding an $\epsilon$-primal-dual stationary point of a  nonconvex-concave minimax problem introduced in Definition \ref{def2}, which will be used as a subproblem solver for the first-order AL method proposed in Section \ref{sec:main}. In particular, we consider the minimax problem
\begin{equation}\label{mmax-prob}
\bH^* = \min_x\max_y\left\{\bH(x,y)\coloneqq \bh(x,y)+\np(x)-\nq(y)\right\}.
\end{equation}
Assume that problem \eqref{mmax-prob} has at least one optimal solution and $p, q$ satisfy Assumption \ref{a1}.  In addition, $\bh$ satisfies the following assumption.

\begin{assumption}\label{mmax-a}
The function $\bh$ is $L_{\nabla\bh}$-smooth on $\dom\,\np\times\dom\,\nq$, and moreover, $\bh(x,\cdot)$ is concave for any $x\in\dom\,\np$.
\end{assumption}

Numerous algorithms have been developed for finding an approximate stationary point of the special case of \eqref{mmax-prob} with $p, q$ being the indicator function of a closed convex set
(e.g., see \cite{jin2020local,lin2020near,nou19,rafique2021weakly,thekumparampil2019efficient,Yang20}). They are however not applicable to \eqref{mmax-prob} in general.
Recently, an accelerated inexact proximal point smoothing (AIPP-S) scheme was proposed in \cite{kong2021accelerated} for finding an approximate stationary point of a class of minimax composite nonconvex optimization problems, which includes \eqref{mmax-prob} as a special case. When applied to \eqref{mmax-prob},  AIPP-S requires the availability of the oracle including exact evaluation of $\nabla_x \bh(x,y)$ and \beq \label{oracle}
\argmin_x \left\{\np(x)+\frac{1}{2\lambda}\|x-x'\|^2\right\}, \qquad \argmax_y \left\{\bh(x',y)-\nq(y)-\frac{1}{2\lambda}\|y-y'\|^2\right\}
\eeq
for any $\lambda>0$, $x'\in\bR^n$ and $y'\in\bR^m$. Notice that $\bh$ is typically sophisticated and the \emph{exact} solution of the second problem in \eqref{oracle} usually cannot be found. As a result, AIPP-S is generally not implementable for \eqref{mmax-prob}, though an operation complexity of $\cO(\epsilon^{-5/2})$, measured by the number of evaluations of the aforementioned oracle, was established in \cite{kong2021accelerated} for it to find an $\epsilon$-primal-dual stationary point of \eqref{mmax-prob}. In addition,  a first-order method was proposed in \cite{zhao2020primal} enjoying an operation complexity of 
$\cO(\varepsilon^{-3}\log\varepsilon^{-1})$, measured by the number of evaluations of $\nabla \bh$ and proximal operator of $p$ and $q$, for finding an $\epsilon$-primal stationary point $x'$ of \eqref{mmax-prob} satisfying
\[
 \Big\|\lambda^{-1}(x'- \argmin_{x} \Big\{\max_y H(x,y)+\frac{1}{2\lambda}\|x-x'\|^2\Big\} \Big\|\leq \epsilon
\]
for some $0<\lambda<L_{\nabla\bh}^{-1}$.  One can see that such $x'$ is an approximate stationary point of  \eqref{mmax-prob} by viewing it as a minimization problem.  Consequently, this method does not suit our need since we aim to find an $\epsilon$-primal-dual stationary point of \eqref{mmax-prob}  introduced in Definition \ref{def2}.

In what follows, we first propose a modified optimal first-order method for solving a strongly-convex-strongly-concave minimax problem in Subsection \ref{strong-cvx-ccv}. Using this method as a subproblem solver for an inexact proximal point scheme, we then propose a first-order method for \eqref{mmax-prob} in Subsection \ref{ppa}, which enjoys an operation complexity of $\cO(\epsilon^{-5/2}\log \epsilon^{-1})$, measured by the number of evaluations of $\nabla \bh$ and proximal operator of $p$ and $q$, for finding an $\epsilon$-primal-dual stationary point of \eqref{mmax-prob}.   

\subsection{A modified optimal first-order method for strongly-convex-strongly-concave minimax problem}
\label{strong-cvx-ccv}

In this subsection, we consider the strongly-convex-strongly-concave minimax problem
\begin{equation}\label{ea-prob}
\H^*=\min_{x}\max_{y}\left\{\H(x,y)\coloneqq \h(x,y)+\np(x)-\nq(y)\right\},
\end{equation}
where $\np, \nq$ satisfy Assumption \ref{a1} and $\h$ satisfies the following assumption.
\begin{assumption}\label{ea}
$\h(x,y)$ is $\sigma_x$-strongly-convex-$\sigma_y$-strongly-concave and $L_{\nabla\h}$-smooth on $\dom\,\np\times\dom\,\nq$ for some $\sigma_x,\sigma_y>0$.
\end{assumption}

Recently, a novel optimal first-order method \cite[Algorithm 4]{kovalev2022first} was proposed for solving 
\eqref{ea-prob}. Though the solution sequence of this method converges to the optimal solution with an optimal rate,  it lacks a verifiable termination criterion and also the approximate solution found by it may never be 
an $\bar\epsilon$-primal-dual stationary point of \eqref{ea-prob} (see Definition \ref{def2}) for a prescribed tolerance $\bar\epsilon>0$.  To tackle these issues, we next propose an optimal first-order method by modifying \cite[Algorithm 4]{kovalev2022first} for finding an approximate primal-dual stationary point of \eqref{ea-prob}. Before proceeding, we introduce some notation below, most of which is adopted from \cite{kovalev2022first}.

Recall that $\mcX=\dom\,p$ and $\mcY=\dom\,q$. Let $(x^*,y^*)$ denote the optimal solution of \eqref{ea-prob}, $z^*=-\sigma_x x^*$,  and
\begin{align}
&D_\bfx\coloneqq \max\{\|u-v\|\big|u,v\in\mcX\},\quad D_\bfy\coloneqq\max\{\|u-v\|\big|u,v\in\mcY\}, \label{mmax-D}\\
&\H_{\rm low}=\min\left\{\H(x,y)|\right(x,y)\in\mcX \times\mcY\}, \label{ea-bnd} \\
&\hat h(x,y)=\h(x,y)-\sigma_x\|x\|^2/2+\sigma_y\|y\|^2/2,\label{ea-hatf}\\
&\cG(z,y)=\sup_{x}\{\langle x,z\rangle-\np(x)-\hat h(x,y)+\nq(y)\},\label{ea-G}\\
&\cP(z,y)=\sigma_x^{-1}\|z\|^2/2+\sigma_y\|y\|^2/2+\cG(z,y),\label{ea-P}\\
&\vartheta_k=\eta_z^{-1}\|z^k-z^*\|^2+\eta_y^{-1}\|y^k-y^*\|^2+2\bar\alpha^{-1}(\cP(z^k_f,y^k_f)-\cP(z^*,y^*)), \label{ea-L} \\
&a^k_x(x,y)=\nabla_x\hat h(x,y)+\sigma_x(x-\sigma_x^{-1}z^k_g)/2,\quad a^k_y(x,y)=-\nabla_y\hat h(x,y)+\sigma_y y+\sigma_x(y-y^k_g)/8,\notag
\end{align}
where $\bar \alpha=\min\left\{1,\sqrt{8\sigma_y/\sigma_x}\right\}$, $\eta_z=\sigma_x/2$, $\eta_y=\min\left\{1/(2\sigma_y),4/(\bar\alpha\sigma_x)\right\}$, and $y^k$, $y^k_f$, $y^k_g$, $z^k$, $z^k_f$ and $z^k_g$ are generated at iteration $k$ of Algorithm \ref{mmax-alg1} below. By Assumptions~\ref{a1} and~\ref{ea}, one can observe that $D_\bfx$, $D_\bfy$ and $\H_{\rm low}$ are finite.

We are now ready to present a modified optimal first-order method for solving \eqref{ea-prob} in Algorithm \ref{mmax-alg1}. It is a slight modification of the novel optimal first-order method \cite[Algorithm 4]{kovalev2022first} by incorporating a forward-backward splitting scheme and also a verifiable termination criterion (see steps 23-25 in Algorithm \ref{mmax-alg1}) in order to find an $\bar\epsilon$-primal-dual stationary point of \eqref{ea-prob} (see Definition \ref{def2}) for any prescribed tolerance $\bar\epsilon>0$. 

\begin{algorithm}[H]
\caption{A modified optimal first-order method for \eqref{ea-prob}}
\label{mmax-alg1}
\begin{algorithmic}[1]
\REQUIRE $\bar\epsilon>0$, $\bar z^0=z^0_f\in-\sigma_x\dom\,\np$,\footnote{} $\bar y^0=y^0_f\in\dom\,\nq$, $(z^0,y^0)=(\bar z^0, \bar y^0)$,  $\bar \alpha=\min\left\{1,\sqrt{8\sigma_y/\sigma_x}\right\}$, $\eta_z=\sigma_x/2$, $\eta_y=\min\left\{1/(2\sigma_y),4/(\bar \alpha\sigma_x)\right\}$, $\beta_t=2/(t+3)$, $\zeta=\left(2\sqrt{5}(1+8L_{\nabla\h}/\sigma_x)\right)^{-1}$, $\gamma_x=\gamma_y=8\sigma_x^{-1}$, and $\bar\zeta=\min\{\sigma_x,\sigma_y\}/L_{\nabla \h}^2$.
\FOR{$k=0,1,2,\ldots$}
\STATE $(z^k_g,y^k_g)=\bar \alpha(z^k,y^k)+(1-\bar \alpha)(z^k_f,y^k_f)$.
\STATE $(x^{k,-1},y^{k,-1})=(-\sigma_x^{-1}z^k_g,y^k_g)$.
\STATE $x^{k,0}=\prox_{\zeta\gamma_x\np}(x^{k,-1}-\zeta\gamma_x a^k_x(x^{k,-1},y^{k,-1}))$.
\STATE $y^{k,0}=\prox_{\zeta\gamma_y \nq}(y^{k,-1}-\zeta\gamma_y a^k_y(x^{k,-1},y^{k,-1}))$.
\STATE $b^{k,0}_x=\frac{1}{\zeta\gamma_x}(x^{k,-1}-\zeta\gamma_x a^k_x(x^{k,-1},y^{k,-1})-x^{k,0})$.
\STATE $b^{k,0}_y=\frac{1}{\zeta\gamma_y}(y^{k,-1}-\zeta\gamma_y a^k_y(x^{k,-1},y^{k,-1})-y^{k,0})$.
\STATE $t=0$.
\WHILE{\\ $\gamma_x\|a^k_x(x^{k,t},y^{k,t})+b^{k,t}_x\|^2+\gamma_y\|a^k_y(x^{k,t},y^{k,t})+b^{k,t}_y\|^2>\gamma_x^{-1}\|x^{k,t}-x^{k,-1}\|^2+\gamma_y^{-1}\|y^{k,t}-y^{k,-1}\|^2$\\~~}
\STATE $x^{k,t+1/2}=x^{k,t}+\beta_t(x^{k,0}-x^{k,t})-\zeta\gamma_x(a^k_x(x^{k,t},y^{k,t})+b^{k,t}_x)$.
\STATE $y^{k,t+1/2}=y^{k,t}+\beta_t(y^{k,0}-y^{k,t})-\zeta\gamma_y(a^k_y(x^{k,t},y^{k,t})+b^{k,t}_y)$.
\STATE $x^{k,t+1}=\prox_{\zeta\gamma_x \np}(x^{k,t}+\beta_t(x^{k,0}-x^{k,t})-\zeta\gamma_x a^k_x(x^{k,t+1/2},y^{k,t+1/2}))$.
\STATE $y^{k,t+1}=\prox_{\zeta\gamma_y \nq}(y^{k,t}+\beta_t(y^{k,0}-y^{k,t})-\zeta\gamma_y a^k_y(x^{k,t+1/2},y^{k,t+1/2}))$.
\STATE $b^{k,t+1}_x=\frac{1}{\zeta\gamma_x}(x^{k,t}+\beta_t(x^{k,0}-x^{k,t})-\zeta\gamma_x a^k_x(x^{k,t+1/2},y^{k,t+1/2})-x^{k,t+1})$.
\STATE $b^{k,t+1}_y=\frac{1}{\zeta\gamma_y}(y^{k,t}+\beta_t(y^{k,0}-y^{k,t})-\zeta\gamma_y a^k_y(x^{k,t+1/2},y^{k,t+1/2})-y^{k,t+1})$.
\STATE $t \leftarrow t+1$.
\ENDWHILE
\STATE $(x^{k+1}_f,y^{k+1}_f)=(x^{k,t},y^{k,t})$.
\STATE $(z^{k+1}_f,w^{k+1}_f)=(\nabla_x\hat h(x^{k+1}_f,y^{k+1}_f)+b^{k,t}_x,-\nabla_y\hat h(x^{k+1}_f,y^{k+1}_f)+b^{k,t}_y)$.
\STATE $z^{k+1}=z^k+\eta_z\sigma_x^{-1}(z^{k+1}_f-z^k)-\eta_z(x^{k+1}_f+\sigma_x^{-1}z^{k+1}_f)$.
\STATE $y^{k+1}=y^k+\eta_y\sigma_y(y^{k+1}_f-y^k)-\eta_y(w^{k+1}_f+\sigma_yy^{k+1}_f)$.
\STATE $x^{k+1}=-\sigma_x^{-1}z^{k+1}$.
\STATE $\tx^{k+1}=\prox_{\bar\zeta \np}(x^{k+1}-\bar\zeta\nabla_x\h(x^{k+1},y^{k+1}))$.
\STATE $\ty^{k+1}=\prox_{\bar\zeta \nq}(y^{k+1}+\bar\zeta\nabla_y\h(x^{k+1},y^{k+1}))$.
\STATE Terminate the algorithm and output $(\tx^{k+1},\ty^{k+1})$ if
\begin{equation}\label{ea-term}
\|\bar\zeta^{-1}(x^{k+1}-\tx^{k+1},\ty^{k+1}-y^{k+1})-(\nabla \h(x^{k+1},y^{k+1})-\nabla \h(\tx^{k+1},\ty^{k+1}))\|\leq\bar\epsilon.
\end{equation}
\ENDFOR
\end{algorithmic}							
\end{algorithm}
\footnotetext{For convenience, $-\sigma_x\dom\,\np$ stands for the set $\{-\sigma_x u|u\in\dom\,\np\}$.}

The following theorem presents \emph{iteration and operation complexity} of Algorithm~\ref{mmax-alg1} for finding an $\bar\epsilon$-primal-dual stationary point of problem \eqref{ea-prob}, whose proof is deferred to Subsection \ref{sec:proof2}.

\begin{thm}[{\bf Complexity of Algorithm \ref{mmax-alg1}}]\label{ea-prop}
Suppose that Assumptions~\ref{a1} and~\ref{ea} hold. Let $\H^*$, $D_\bfx$, $D_\bfy$, $\H_{\rm low}$, and $\vartheta_0$ be defined in \eqref{ea-prob}, \eqref{mmax-D}, \eqref{ea-bnd} and \eqref{ea-L}, $\sigma_x$, $\sigma_y$ and $L_{\nabla \h}$ be given in Assumption \ref{ea}, $\bar \alpha$, $\eta_y$, $\eta_z$, $\bar\epsilon$, $\bar\zeta$ be given in Algorithm~\ref{mmax-alg1}, and 
\begin{align}
\bar \delta=&\ (2+\bar \alpha^{-1})\sigma_x D_\bfx^2+\max\{2\sigma_y,\bar \alpha\sigma_x/4\}D_\bfy^2,\label{ea-tP}\\
\bar K=&\ \left\lceil\max\left\{\frac{2}{\bar \alpha},\frac{\bar \alpha\sigma_x}{4\sigma_y}\right\}\log\frac{4\max\{\eta_z\sigma_x^{-2},\eta_y\}\vartheta_0}{(\bar\zeta^{-1}+ L_{\nabla \h})^{-2}\bar\epsilon^2}\right\rceil_+, \label{ea-K} \\
\bar N=&\ \left\lceil\max\left\{2,\sqrt{\frac{\sigma_x}{2\sigma_y}}\right\}\log\frac{4\max\left\{1/(2\sigma_x),\min\left\{1/(2\sigma_y),4/(\bar \alpha\sigma_x)\right\}\right\}\left(\bar \delta+2\bar \alpha^{-1}\left(\H^*-\H_{\rm low}\right)\right)}{(L_{\nabla \h}^2/\min\{\sigma_x,\sigma_y\}+ L_{\nabla \h})^{-2}\bar\epsilon^2}\right\rceil_+\notag\\
&\ \times\ \left(\left\lceil96\sqrt{2}\left(1+8L_{\nabla \h}\sigma_x^{-1}\right)\right\rceil+2\right).\label{ea-N}
\end{align}
Then Algorithm~\ref{mmax-alg1} outputs an $\bar\epsilon$-primal-dual stationary point of \eqref{ea-prob} in at most $\bar K$ iterations. Moreover, the total number of evaluations of $\nabla \h$ and proximal operator of $\np$ and $\nq$ performed in Algorithm~\ref{mmax-alg1} is no more than $\bar N$, respectively.
\end{thm}

\begin{rem}
It can be observed from Theorem \ref{ea-prop} that Algorithm~\ref{mmax-alg1} enjoys an operation complexity of $\cO(\log(1/ \bar\epsilon))$, measured by the number of evaluations of $\nabla \h$ and proximal operator of $\np$ and $\nq$, for finding an $\bar\epsilon$-primal-dual stationary point of the  strongly-convex-strongly-concave minimax problem \eqref{ea-prob}.
\end{rem}

\subsection{A first-order method for problem \eqref{mmax-prob}} 
\label{ppa}

In this subsection, we propose a first-order method for finding an $\epsilon$-primal-dual stationary point of problem \eqref{mmax-prob} (see Definition \ref{def2}) for any prescribed tolerance $\epsilon>0$. In particular, we first add a perturbation to the max part of \eqref{mmax-prob} for obtaining an approximation of \eqref{mmax-prob}, which is given as follows:
\beq\label{mmax-prob-approx}
 \min_x\max_y\left\{\bh(x,y)+\np(x)-\nq(y)-\frac{\epsilon}{4D_\bfy}\|y-\hat y^0\|^2\right\}
\eeq
for some $\hat y^0\in\dom\,\nq$, where $D_\bfy$ is given in \eqref{mmax-D}. We then apply an inexact proximal point method \cite{kaplan1998proximal} to \eqref{mmax-prob-approx}, which consists of approximately solving a sequence of subproblems
\beq\label{ppa-subprob}
 \min_x\max_y\left\{\bH_k(x,y):=\bh_k(x,y)+\np(x)-\nq(y)\right\}, 
 \eeq
 where 
 \begin{equation}\label{mmax-sub}
\bh_k(x,y)=\bh(x,y)-\epsilon\|y-\hat y^0\|^2/(4D_\bfy)+L_{\nabla \bh}\|x-x^k\|^2.
\end{equation}
By Assumption \ref{mmax-a}, one can observe that (i) $\bh_k$ is $L_{\nabla \bh}$-strongly convex in $x$ and $\epsilon/(2D_\bfy)$-strongly concave in $y$ on $\dom\,\np\times\dom\,\nq$; (ii) $\bh_k$ is $(3L_{\nabla \bh}+\epsilon/(2D_\bfy))$-smooth on $\dom\,\np\times\dom\,\nq$. Consequently, problem \eqref{ppa-subprob} is a special case of \eqref{ea-prob} and can be suitably solved by Algorithm~\ref{mmax-alg1}. The resulting first-order method for \eqref{mmax-prob} is presented in Algorithm \ref{mmax-alg2}.

\begin{algorithm}[H]
\caption{A first-order method for problem~\eqref{mmax-prob}}
\label{mmax-alg2}
\begin{algorithmic}[1]
\REQUIRE $\epsilon>0$, $\hat\epsilon_0\in(0,\epsilon/2]$, 
$(\hat x^0,\hat y^0)\in\dom\,\np\times\dom\,\nq$, $(x^0,y^0)=(\hat x^0,\hat y^0)$, 
and $\hat\epsilon_k=\hat\epsilon_0/(k+1)$.
\FOR{$k=0,1,2,\ldots$}
\STATE Call Algorithm~\ref{mmax-alg1} with $\h\leftarrow \bh_k$, $\bar\epsilon \leftarrow \hat\epsilon_k$, $\sigma_x\leftarrow L_{\nabla \bh}$, $\sigma_y\leftarrow \epsilon/(2D_\bfy)$, $L_{\nabla \h}\leftarrow 3L_{\nabla \bh}+\epsilon/(2D_\bfy)$, $\bar z^0=z^0_f\leftarrow-\sigma_x x^k$, $\bar y^0=y^0_f\leftarrow y^k$, and denote its output by $(x^{k+1},y^{k+1})$, where $\bh_k$ is given in \eqref{mmax-sub}.
\STATE Terminate the algorithm and output $(\xe,\ye)=(x^{k+1},y^{k+1})$ if
\begin{equation}\label{mmax-term}
\|x^{k+1}-x^k\|\leq\epsilon/(4L_{\nabla \bh}).
\end{equation}
\ENDFOR
\end{algorithmic}
\end{algorithm}

\begin{rem}\label{ppa-rem}
It is seen from step 2 of Algorithm~\ref{mmax-alg2} that $(x^{k+1},y^{k+1})$ results from applying Algorithm~\ref{mmax-alg1} to the subproblem \eqref{ppa-subprob}. As will be shown in Lemma~\ref{innercplx-alg6}, $(x^{k+1},y^{k+1})$ is an $\hat\epsilon_k$-primal-dual stationary point of \eqref{ppa-subprob}.   
\end{rem}

We next study complexity of Algorithm \ref{mmax-alg2} for finding an $\epsilon$-primal-dual stationary point of problem~\eqref{mmax-prob}. Before proceeding, we define
\beq
\bH_{\rm low}:=\min\left\{\bH(x,y)|(x,y)\in\dom\,\np\times\dom\,\nq\right\}.\label{mmax-bnd}
\eeq
By Assumption~\ref{a1}, one can observe that $\bH_{\rm low}$ is finite.

The following theorem presents \emph{iteration and operation complexity} of Algorithm~\ref{mmax-alg2} for finding an $\epsilon$-primal-dual stationary point of problem \eqref{mmax-prob}, whose proof is deferred to Subsection \ref{sec:proof2-2}.

\begin{thm}[{\bf Complexity of Algorithm \ref{mmax-alg2}}]\label{mmax-thm}
Suppose that Assumption~\ref{mmax-a} holds. Let $\bH^*$, $H$ $D_\bfx$, $D_\bfy$, and $\bH_{\rm low}$ be defined in \eqref{mmax-prob}, 
\eqref{mmax-D} and \eqref{mmax-bnd}, $L_{\nabla \bh}$ be given in Assumption \ref{mmax-a}, $\epsilon$, $\hat\epsilon_0$ and $\hat x^0$ be given in Algorithm~\ref{mmax-alg2}, and 
\begin{align}
\halpha=&\ \min\left\{1,\sqrt{4\epsilon/(D_\bfy L_{\nabla \bh})}\right\},\label{mmax-balpha}\\
\hdelta=&\ (2+\halpha^{-1})L_{\nabla \bh} D_\bfx^2+\max\left\{\epsilon/D_\bfy,\halpha L_{\nabla \bh}/4\right\}D_\bfy^2,\label{mmax-tP}\\
\wT=&\ \left\lceil16(\max_y\bH(\hat x^0,y)-\bH^*+\epsilon D_\bfy/4)L_{\nabla \bh}\epsilon^{-2}+32\hat\epsilon_0^2(1+4D_\bfy^2L_{\nabla \bh}^2\epsilon^{-2})\epsilon^{-2}-1\right\rceil_+,\label{mmax-K}\\
\widehat N=&\ \left(\left\lceil96\sqrt{2}\left(1+\left(24L_{\nabla \bh}+4\epsilon/D_\bfy\right)L_{\nabla \bh}^{-1}\right)\right\rceil+2\right)\max\left\{2,\sqrt{D_\bfy L_{\nabla \bh}\epsilon^{-1}}\right\}\notag\\
&\ \times\Bigg((\wT+1)\Bigg(\log\frac{4\max\left\{\frac{1}{2L_{\nabla \bh}},\min\left\{\frac{D_\bfy}{\epsilon},\frac{4}{\halpha L_{\nabla \bh}}\right\}\right\}\left(\hdelta+2\halpha^{-1}(\bH^*-\bH_{\rm low}+\epsilon D_\bfy/4+L_{\nabla \bh} D_\bfx^2)\right)}{\left[(3L_{\nabla \bh}+\epsilon/(2D_\bfy))^2/\min\{L_{\nabla \bh},\epsilon/(2D_\bfy)\}+ 3L_{\nabla \bh}+\epsilon/(2D_\bfy)\right]^{-2}\hat\epsilon_0^2}\Bigg)_+\notag\\
&\ +\wT+1+2\wT\log(\wT+1) \Bigg).\label{mmax-N-old}
\end{align}
Then Algorithm~\ref{mmax-alg2} terminates and outputs an $\epsilon$-primal-dual stationary point $(\xe,\ye)$ of \eqref{mmax-prob} in at most $\wT+1$ outer iterations that satisfies 
\begin{equation}\label{upperbnd-old}
\max_y\bH(\xe,y)\leq \max_y\bH(\hat x^0,y)+\epsilon D_\bfy/4+2\hat\epsilon_0^2\left(L_{\nabla \bh}^{-1}+4D_\bfy^2L_{\nabla \bh}\epsilon^{-2}\right).
\end{equation}
Moreover, the total number of evaluations of $\nabla \bh$ and proximal operator of $p$ and $q$ performed in Algorithm~\ref{mmax-alg2} is no more than $\widehat N$, respectively.
\end{thm}

\begin{rem}
Since $\hat\epsilon_0\in(0,\epsilon/2]$, one can observe from Theorem \ref{mmax-thm} that $\halpha=\cO(\epsilon^{1/2})$, $\hdelta=\cO(\epsilon^{-1/2})$, $\wT=\cO(\epsilon^{-2})$, and $\widehat N=\cO(\epsilon^{-5/2}\log(\hat\epsilon_0^{-1}\epsilon^{-1}))$. Consequently,  Algorithm~\ref{mmax-alg2} enjoys an operation complexity  of $\cO(\epsilon^{-5/2}\log(\hat\epsilon_0^{-1} \epsilon^{-1}))$, measured by the number of evaluations of $\nabla \bh$ and proximal operator of $p$ and $q$, for finding an $\epsilon$-primal-dual stationary point of the nonconvex-concave minimax problem \eqref{mmax-prob}.
\end{rem}

\section{A first-order augmented Lagrangian method for problem \eqref{prob}}\label{sec:main}

In this section, we propose a first-order augmented Lagrangian (FAL) method for problem \eqref{prob}, and study its complexity for finding an approximate KKT point of \eqref{prob}.

One standard approach for solving constrained nonlinear program is to solve a sequence of unconstrained nonlinear program problems, which are typically penalty or augmented Lagrangian subproblems (e.g., see \cite{nocedal1999numerical}). In a similar spirit,  we next propose an FAL method in Algorithm \ref{AL-alg} for solving \eqref{prob}. In particular, at each iteration, the FAL method finds an approximate primal-dual stationary point of an AL subproblem in the form of
\beq\label{AL-sub0}
\min_x\max_y\AL(x,y,\lambda_\bfx,\lambda_\bfy;\rho),
\eeq
where $\AL$ is the AL function associated with problem \eqref{prob} defined in \eqref{AL}, $\lambda_\bfx\in\bR_+^{\tn}$ and $\lambda_\bfy\in\bR_+^{\tm}$ are a Lagrangian multiplier estimate, and $\rho>0$ is a penalty parameter, which are updated by a standard scheme.
In view of Assumption \ref{a1}, one can observe that $\AL$ enjoys the following nice structure.
\bi
\item 
For any given $\rho>0$, $\lambda_\bfx\in\bR_+^{\tn}$ and $\lambda_\bfy\in\bR_+^{\tm}$, 
$\AL$ is the sum of smooth function $f(x,y)+\left(\|[\lambda_\bfx+\rho c(x)]_+\|^2-\|\lambda_\bfx\|^2\right)/(2\rho)-\left(\|[\lambda_\bfy+\rho d(x,y)]_+\|^2-\|\lambda_\bfy\|^2\right)/(2\rho)$ with Lipschitz continuous gradient and possibly nonsmooth function $p(x)-q(y)$ with exactly computable proximal operator.
\item $\AL$ is nonconvex in $x$ but concave in $y$.
\ei
Thanks to the above nice structure of $\AL$,  we will use Algorithm \ref{mmax-alg2} as a solver to find an approximate primal-dual stationary point of the AL subproblem \eqref{AL-sub0}.

Recall that $\mcX=\dom\,p$ and $\mcY=\dom\,q$. Before presenting an FAL method for \eqref{prob}, we let
\begin{align}
&\AL_\bfx(x,y,\lambda_\bfx;\rho):=F(x,y)+\frac{1}{2\rho}\left(\|[\lambda_\bfx+\rho c(x)]_+\|^2-\|\lambda_\bfx\|^2\right), \label{x-AL} \\
&c_{\rm hi}:=\max\{\|c(x)\|\big|x\in\mcX\},\quad d_{\rm hi}:=\max\{\|d(x,y)\|\big|(x,y)\in\mcX\times\mcY\}, \label{cdhi} 
\end{align}
where $\AL_\bfx(\cdot,y,\lambda_\bfx;\rho)$ can be viewed as the AL function for the minimization part of \eqref{prob}, namely, the problem $\min_x \{F(x,y)| c(x)\leq 0\}$ for any $y\in\mcY$. Besides, we make one additional assumption below regarding the availability of a nearly feasible point for the minimization part of \eqref{prob}. Due to the possible nonconvexity of $c_i$'s, it will be used to specify an initial point for solving the AL subproblems (see step 2 of Algorithm \ref{AL-alg}) so that the resulting FAL method outputs an approximate KKT point of  \eqref{prob} nearly satisfying the constraint $c(x) \leq 0$.

\begin{assumption}\label{knownfeas}		
For any given $\varepsilon\in (0,1)$, a $\sqrt{\varepsilon}$-nearly feasible point $x_\bff$ of problem~\eqref{prob}, namely $x_\bff\in\mcX$ satisfying $\|[c(x_\bff)]_+\|\le\sqrt{\varepsilon}$, can be found.
\end{assumption}

\begin{rem}
A very similar assumption as Assumption~\ref{knownfeas} was considered in \cite{CGLY17,GY19,LZ12,XW19}. In addition,  when the error bound condition $\|[c(x)]_+\|=\cO(\dist(0,\partial (\|[c(x)]_+\|^2+\delta_{\mcX}(x))))^\nu)$ holds on a level set of $\|[c(x)]_+\|$ for some $\nu>0$, Assumption~\ref{knownfeas} holds for problem~\eqref{prob} (e.g., see \cite{lu2022single,S19iAL}). In this case, one can find the above $x_\bff$ by applying a projected gradient method to the problem $\min_{x\in\mcX}\|[c(x)]_+\|^2$.
\end{rem}

 We are now ready to present an FAL method for solving  problem \eqref{prob}.
 
\begin{algorithm}[H]
\caption{A first-order augmented Lagrangian method for problem \eqref{prob}}\label{AL-alg}
\begin{algorithmic}[1]
\REQUIRE $\varepsilon, \tau\in(0,1)$,  $\epsilon_k=\tau^k$, $\rho_k=\epsilon_k^{-1}$, $\Lambda>0$, $\lambda_\bfx^0\in\cB^+_\Lambda$, $\lambda_\bfy^0\in\bR_+^{\tm}$, $(x^0,y^0)\in\dom\,p\times\dom\,q$, 
and $x_\bff\in\dom\,p$ with $\|[c(x_\bff)]_+\|\leq\sqrt{\varepsilon}$ (see Assumption \ref{knownfeas}).	
\FOR{$k=0,1,\dots$}
\STATE Set
\begin{align} \label{xinit}
x^k_{\rm init}=\left\{\begin{array}{ll}
x^k,&\quad \mbox{if }\AL_\bfx(x^k,y^k,\lambda^k_\bfx;\rho_k)\leq\AL_\bfx(x_\bff,y^k,\lambda^k_\bfx;\rho_k),\\
x_\bff,&\quad \mbox{otherwise.}
\end{array}\right.
\end{align}
\STATE Call Algorithm \ref{mmax-alg2} with $\epsilon\leftarrow\epsilon_k$, $\hat\epsilon_0\leftarrow\epsilon_k/(2\sqrt{\rho_k})$, $(x^0,y^0)\leftarrow (x^k_{\rm init},y^k)$ and $L_{\nabla h}\leftarrow L_k$ to find an $\epsilon_k$-primal-dual stationary point $(x^{k+1},y^{k+1})$ of 
\beq\label{AL-sub}
\min_x\max_y\AL(x,y,\lambda^k_\bfx,\lambda^k_\bfy;\rho_k)
\eeq
 where
\beq\label{Lk}
L_k= L_{\nabla f}+\rho_kL_c^2+\rho_kc_{\rm hi}L_{\nabla c}+\|\lambda^k_\bfx\|L_{\nabla c}+\rho_kL_d^2+\rho_kd_{\rm hi}L_{\nabla d}+\|\lambda^k_\bfy\|L_{\nabla d}.
\eeq
\STATE Set $\lambda_\bfx^{k+1}=\Pi_{\cB^+_\Lambda}(\lambda^k_\bfx+\rho_kc(x^{k+1}))$ and $\lambda^{k+1}_\bfy=[\lambda^k_\bfy+\rho_kd(x^{k+1},y^{k+1})]_+$.
\STATE If $\epsilon_k\leq\varepsilon$, terminate the algorithm and output $(x^{k+1},y^{k+1})$. 
\ENDFOR
\end{algorithmic}
\end{algorithm}
\begin{rem}
\bi
\item[(i)] $\lambda^{k+1}_\bfx$ results from projecting onto a nonnegative Euclidean ball the standard Lagrangian multiplier estimate $\tilde\lambda_\bfx^{k+1}$ obtained by the classical scheme $\tilde\lambda_\bfx^{k+1}=[\lambda^k_\bfx+\rho_kc(x^{k+1})]_+$. It is called a safeguarded Lagrangian multiplier in the relevant literature \cite{BM14,BM20,KS17example}, which has been shown to enjoy many practical and theoretical advantages (see \cite{BM14} for discussions).
\item[(ii)] In view of Theorem \ref{mmax-thm}, one can see that an $\epsilon_k$-primal-dual stationary point of \eqref{AL-sub} can be successfully found in step 3 of Algorithm~\ref{AL-alg} by applying Algorithm \ref{mmax-alg2} to problem \eqref{AL-sub}. Consequently, Algorithm~\ref{AL-alg} is well-defined.
\ei
\end{rem}

\subsection{Complexity results for Algorithm \ref{AL-alg}}\label{sec:complexity}

In this subsection we study iteration and operation complexity for Algorithm \ref{AL-alg}. Recall that $\mcX=\dom\,p$ and $\mcY=\dom\,q$. 
Before proceeding, we make one additional assumption below that a generalized Mangasarian-Fromowitz constraint qualification (GMFCQ) holds for the minimization part of \eqref{prob}, a uniform Slater's condition holds for the maximization part of \eqref{prob}, and $F(\cdot,y)$ is Lipschitz continuous on $\mcX$ for any $y\in\mcY$. Specifically, GMFCQ and the Lipschitz continuity of $F(\cdot,y)$ will be used to bound the amount of violation on feasibility and complementary slackness by $(x^{k+1}, \tl^{k+1}_\bfx)$ for the minimization part of \eqref{prob} with $\tl^{k+1}_\bfx=[\lambda^k_\bfx+\rho_kc(x^{k+1})]_+$ (see Lemma \ref{l-xcnstr2}). Likewise, the uniform Slater's condition will be used to bound the amount of violation on feasibility and complementary slackness by $(x^{k+1}, y^{k+1},\lambda^{k+1}_\bfy)$ for the maximization part of \eqref{prob} (see Lemmas \ref{l-ycnstr} and \ref{l-subdcnstr}). 

\begin{assumption}\label{mfcq}
\bi
\item[(i)] There exist some constants $\delta_c$, $\theta>0$ such that for each $x\in\cF(\theta)$ there exists some $v_x \in\mcT_{\mcX}(x)$  satisfying $\|v_x\|=1$ and $v^T_x\nabla c_i(x)\leq-\delta_c$ for all $i\in\cA(x;\theta)$, where $\mcT_{\mcX}(x)$ is the tangent cone of $\mcX$ at $x$, and
\beq\label{def-cAcS}
\cF(\theta)=\{x\in\mcX\big|\|[c(x)]_+\|_\infty \leq\theta\},\quad\cA(x;\theta)=\{i|c_i(x)\geq-\theta,\ 1\leq i\leq \tn\}.
\eeq
\item[(ii)] For each $x\in\mcX$, there exists some $\hat y_x\in\mcY$ such that $d_i(x,\hat y_x)<0$ for all $i=1,2,\dots,\tm$, and moreover, $\delta_d:=\inf\{-d_i(x,\hat y_x)|x\in\mcX,\ i=1,2,\dots,\tm\}>0$.
\item[(iii)]  $F(\cdot,y)$ is $L_F$-Lipschitz continuous on $\mcX$ for any $y\in\mcY$.
\ei
\end{assumption}
\begin{rem}
\bi
\item[(i)] Assumption \ref{mfcq}(i) can be viewed as a robust counterpart of MFCQ.  It implies that MFCQ holds for all the minimization problems, resulting from the minimization part of \eqref{prob} by fixing  $y\in\mcY$ and perturbing $c_i(x)$ at most by $\theta$.
\item[(ii)] The latter part of Assumption \ref{mfcq}(ii) can be weakened to the one that the pointwise Slater's condition holds for the constraint on $y$ in \eqref{prob}, that is, there exists $\hat y_x\in\mcY$ such that $d(x,\hat y_x)<0$ for each $x\in\mcX$. Indeed, if $\delta_d>0$, Assumption \ref{mfcq}(ii) holds. Otherwise, one can solve the perturbed counterpart of \eqref{prob} with $d(x,y)$ being replaced by $d(x,y)-\epsilon$ for some suitable $\epsilon>0$ instead, which satisfies Assumption \ref{mfcq}(ii).
\item[(iii)] In view of Assumption \ref{a1},  one can observe that  if $p$ is Lipschitz continuous on $\mcX$, $F(\cdot,y)$ is Lipschitz continuous on $\mcX$ for any $y\in\mcY$.  Thus, Assumption \ref{mfcq}(iii) is mild.
\ei
\end{rem}
In order to characterize the approximate solution found by Algorithm \ref{AL-alg}, we next introduce a notion called an $\varepsilon$-KKT solution of problem \eqref{prob}.

One can observe from Lemma \ref{dual-bnd}(iii) in Subsection \ref{sec4.3} that problem \eqref{prob} is equivalent to
\[
\min_{x,\lambda_\bfy}\big\{\max_y F(x,y)-\langle\lambda_\bfy,d(x,y)\rangle+\delta_{\bR^{\tm}_+}(\lambda_\bfy) \big|c(x)\leq0\big\}.
\]
By this, one can further see that problem \eqref{prob} is equivalent to
\[
\min_{x,\lambda_\bfy}\max_{\lambda_\bfx}\Big\{\max_y \{F(x,y)-\langle\lambda_\bfy,d(x,y)\rangle+\delta_{\bR^{\tm}_+}(\lambda_\bfy)\} +\langle\lambda_\bfx,c(x)\rangle-\delta_{\bR^{\tn}_+}(\lambda_\bfx)\Big\},
\]
which is a nonconvex-concave minimax problem
\beq \label{prob-ref}
\min_{x,\lambda_\bfy}\max_{y,\lambda_\bfx} \left\{F(x,y)+\langle\lambda_\bfx,c(x)\rangle-\langle\lambda_\bfy,d(x,y)\rangle -\delta_{\bR^{\tn}_+}(\lambda_\bfx)+\delta_{\bR^{\tm}_+}(\lambda_\bfy)\right\}.
\eeq
 It follows from \cite[Theorem 3.1]{dai2020optimality} that if $(x,y,\lambda_\bfx,\lambda_\bfy)\in\bR^n \times\bR^m\times\bR^{\tn}_+\times\bR^{\tm}_+$ is a local minimax point of problem \eqref{prob-ref}, then it must also be a primal-dual stationary point of \eqref{prob-ref}. This, combined with Definition \ref{def2}, implies that  $(x,y,\lambda_\bfx,\lambda_\bfy)$ is a KKT point of \eqref{prob-ref} satisfying the conditions:
\begin{align}
& 0\in \partial_x F(x,y)+\nabla c(x)\lambda_\bfx-\nabla_x d(x,y) \lambda_\bfy,  \label{kkt1} \\
& 0\in \partial_y F(x,y)-\nabla_y d(x,y) \lambda_\bfy, \label{kkt2} \\
& c(x) \leq 0, \quad \langle \lambda_\bfx, c(x) \rangle=0, \label{kkt3} \\
& d(x,y) \leq 0, \quad \langle \lambda_\bfy, d(x,y) \rangle=0. \label{kkt4}
\end{align}
Based on this observation and the equivalence of \eqref{prob} and \eqref{prob-ref}, 
we introduce an (approximate) KKT solution for problem \eqref{prob} below.

\begin{defi} \label{approx-kkt-pt}
The pair $(x,y)$ is said to be a KKT solution of problem \eqref{prob} if there exists $(\lambda_\bfx,\lambda_\bfy)\in\bR^{\tn}_+\times\bR^{\tm}_+$ such that the conditions \eqref{kkt1}-\eqref{kkt4} hold. In addition, for any $\varepsilon>0$, $(x,y)$ is said to be an $\varepsilon$-KKT point of problem \eqref{prob} if there exists $(\lambda_\bfx,\lambda_\bfy)\in\bR^{\tn}_+\times\bR^{\tm}_+$ such that
\begin{align*}
& \dist(0, \partial_x F(x,y)+\nabla c(x)\lambda_\bfx-\nabla_x d(x,y) \lambda_\bfy) \leq \varepsilon,   \\
& \dist(0, \partial_y F(x,y)-\nabla_y d(x,y) \lambda_\bfy) \leq \varepsilon,   \\
& \|[c(x)]_+\| \leq  \varepsilon, \quad |\langle \lambda_\bfx, c(x) \rangle|\leq \varepsilon,  \\
& \|[d(x,y)]_+\| \leq  \varepsilon, \quad |\langle \lambda_\bfy, d(x,y) \rangle| \leq \varepsilon. 
\end{align*}
\end{defi} 

Recall that $\mcX=\dom\,p$ and $\mcY=\dom\,q$. To study complexity of Algorithm \ref{AL-alg}, we define
\begin{align}
&f^*(x):=\max\{F(x,y)|d(x,y)\leq0\},\label{fstarx}\\
&F_{\rm hi}:=\max\{F(x,y)|(x,y)\in\mcX\times\mcY\},\quad F_{\rm low}:=\min\{F(x,y)|(x,y)\in\mcX\times\mcY\},\label{Fhi}\\
&\Delta:=F_{\rm hi}-F_{\rm low}, \quad  r:=2\delta_d^{-1}\Delta,\label{def-r} \\
&K:=\left\lceil\log\varepsilon/\log\tau\right\rceil_+, \quad \bbK:=\{0,1,\ldots, K+1\}, \label{K1} 
\end{align}
where $\delta_d$ is given in Assumption \ref{mfcq}, and $\varepsilon$ and $\tau$ are some input parameters of Algorithm \ref{AL-alg}. For convenience, we define $\bbK-1=\{k-1| k\in\bbK\}$. One can observe from Assumption \ref{a1} that $F_{\rm hi}$ and $F_{\rm low}$ are finite. 
Besides, one can easily observe that 
\beq \label{F-gap}
f^*(x)\geq F_{\rm low}, \   F(x,y)-f^*(x)\leq \Delta \quad \forall x\in\mcX, y\in\mcY.
\eeq

We are now ready to present an \emph{iteration and operation complexity} of Algorithm~\ref{AL-alg} for finding an $\cO(\varepsilon)$-KKT solution of problem \eqref{prob}, whose proof is deferred to Section \ref{sec:proofs}.

\begin{thm}\label{complexity}
Suppose that Assumptions \ref{a1}, \ref{knownfeas} and \ref{mfcq} hold. Let $\{(x^k,y^k,\lambda^k_\bfx,\lambda^k_\bfy)\}_{k\in\bbK}$ be generated by Algorithm \ref{AL-alg}, $D_\bfx$, $D_\bfy$, $c_{\rm hi}$, $d_{\rm hi}$, $\Delta$ and $K$ be defined in  \eqref{mmax-D}, \eqref{cdhi},  \eqref{def-r} and \eqref{K1}, $L_F$, $L_{\nabla f}$, $L_{\nabla d}$, $L_{\nabla c}$, $L_c$, $L_{\nabla d}$, $L_d$, $\delta_c$, $\delta_d$ and $\theta$ be given in Assumptions \ref{a1} and \ref{mfcq}, $\varepsilon$,  $\tau$, $\Lambda$ and $\lambda_\bfy^0$ be given in Algorithm \ref{AL-alg}, and
\begin{align}
&L= L_{\nabla f}+L_c^2+c_{\rm hi}L_{\nabla c}+\Lambda L_{\nabla c}+L_d^2+d_{\rm hi}L_{\nabla d}+L_{\nabla d}\sqrt{\|\lambda_\bfy^0\|^2+\frac{2(\Delta+D_\bfy)}{1-\tau}},\label{hL}\\
&\alpha=\min\left\{1, \sqrt{4/(D_\bfy L)}\right\},\quad \delta= (2+\alpha^{-1})L D_\bfx^2+\max\{1/D_\bfy,L/4\}D_\bfy^2,\label{ho}\\
&M=16\max\left\{1/(2L_c^2),4/(\alpha L_c^2)\right\}\left[(3L+1/(2D_\bfy))^2/\min\{L_c^2,1/(2D_\bfy)\}+ 3L+1/(2D_\bfy)\right]^2\nn\\
&\ \ \ \ \ \ \times\left(\delta+2\alpha^{-1}\Big(\Delta+\frac{\Lambda^2}{2}+\frac{3}{2}\|\lambda_\bfy^0\|^2+\frac{3(\Delta+D_\bfy)}{1-\tau}+\rho_kd_{\rm hi}^2+\frac{D_\bfy}{4}+L D_\bfx^2\Big)\right),\label{hM}\\
&T= \Bigg\lceil16L\left(2\Delta+\Lambda+\frac{1}{2}(\tau^{-1}+\|\lambda_\bfy^0\|^2)+\frac{\Delta+D_\bfy}{1-\tau}+\frac{\Lambda^2}{2}+\frac{D_\bfy}{4}\right)+8(1+4D_\bfy^2L^2)\Bigg\rceil_+, \label{hT}\\
&\tlambda^{K+1}_\bfx = [\lambda^K_\bfx+\tau^{-K}c(x^{K+1})]_+.\label{tlx}
\end{align}
Suppose that 
\begin{align}
\varepsilon^{-1} \geq  \max\Bigg\{&\theta^{-1}\Lambda, \theta^{-2}\Big\{4\Delta+2\Lambda+\tau^{-1}+\|\lambda_\bfy^0\|^2+\frac{2(\Delta+D_\bfy)}{1-\tau} +\frac{D_\bfy}{2} +L_c^{-2} +4D_\bfy^2L+\Lambda^2\Big\}, \nn \\ 
& \frac{4\|\lambda_\bfy^0\|^2}{\delta_d^2\tau}+\frac{8(\Delta+D_\bfy)}{\delta_d^2\tau(1-\tau)}\Bigg\}. \label{cond}
\end{align}
Then the following statements hold.
\begin{enumerate}[label=(\roman*)]
\item Algorithm \ref{AL-alg} terminates after $K+1$ outer iterations and outputs an approximate stationary point $(x^{K+1},y^{K+1})$ of \eqref{prob} satisfying
\begin{align}
&\dist(0,\partial_x F(x^{K+1},y^{K+1})+\nabla c(x^{K+1})\tl_x^{K+1}-\nabla_xd(x^{K+1},y^{K+1})\lambda^{K+1}_\bfy) \leq \varepsilon, \label{t1-1} \\
& \dist\left(0,\partial_y F(x^{K+1},y^{K+1})-\nabla_y d(x^{K+1},y^{K+1})\lambda^{K+1}_\bfy\right)\leq\varepsilon, \label{t1-2} \\
&\|[c(x^{K+1})]_+\|\leq \varepsilon\delta_c^{-1}\left(L_F +2L_d\delta_d^{-1}(\Delta+D_\bfy)+1\right), \label{t1-3} \\
& |\langle\tl^{K+1}_\bfx,c(x^{K+1})\rangle| \leq \varepsilon \delta_c^{-1}(L_F +2L_d\delta_d^{-1}(\Delta+D_\bfy)+1)\nn\\
&\qquad \qquad \qquad\qquad \ \
\times\max\{\delta_c^{-1}(L_F +2L_d\delta_d^{-1}(\Delta+D_\bfy)+1), \Lambda\}, \label{t1-4} \\
&\|[d(x^{K+1},y^{K+1})]_+\|\leq2\varepsilon\delta_d^{-1}(\Delta+D_\bfy),\label{t1-5}\\
& |\langle \lambda^{K+1}_\bfy, d(x^{K+1},y^{K+1})\rangle| \leq 2\varepsilon\delta_d^{-1}(\Delta+D_\bfy)\max\{2\delta_d^{-1} (\Delta+D_\bfy), \|\lambda_\bfy^0\|\}. \label{t1-6}
\end{align}
\item The total number of evaluations of $\nabla f$, $\nabla c$, $\nabla d$ and proximal operator of $p$ and $q$ performed in Algorithm \ref{AL-alg} is at most $N$, respectively, where
\begin{align}
N=&\ \left(\left\lceil96\sqrt{2}\left(1+\left(24L+4/D_\bfy\right)/L_c^2\right)\right\rceil+2\right)\max\left\{2,\sqrt{D_\bfy L}\right\}T(1-\tau^4)^{-1}\nn\\
&\ \times(\tau\varepsilon)^{-4}\left(28K\log(1/\tau)+2(\log M)_++2+2\log(2T) \right).\label{N2}
\end{align}
\end{enumerate}
\end{thm}

\begin{rem} 
\bi
\item[(i)] The condition \eqref{cond} on $\varepsilon$ is to ensure that the final penalty parameter $\rho_K$ in Algorithm \ref{AL-alg} is large enough so that  feasibility and complementarity slackness are nearly satisfied at $(x^{K+1},y^{K+1},\tl^{K+1}_\bfx, \lambda^{K+1}_\bfy)$. 

\item[(ii)] One can observe from Theorem \ref{complexity} that Algorithm \ref{AL-alg} enjoys an iteration complexity of $\cO(\log\varepsilon^{-1})$ and an operation complexity of $\cO(\varepsilon^{-4}\log\varepsilon^{-1})$, measured by the number of evaluations of $\nabla f$, $\nabla c$, $\nabla d$ and proximal operator of $p$ and $q$, for finding an $\cO(\varepsilon)$-KKT solution $(x^{K+1},y^{K+1})$ of \eqref{prob} such that
\begin{align*}
& \dist\left(\partial_x F(x^{K+1},y^{K+1})+\nabla c(x^{K+1})\tl_\bfx-\nabla_x d(x^{K+1},y^{K+1}) \lambda_\bfy^{K+1}\right) \leq \varepsilon,\\
& \dist\left(\partial_y F(x^{K+1},y^{K+1})-\nabla_y d(x^{K+1},y^{K+1}) \lambda_\bfy^{K+1}\right)\leq \varepsilon,\\
& \|[c(x^{K+1})]_+\|=\cO(\varepsilon), \quad |\langle \tl_\bfx^{K+1}, c(x^{K+1}) \rangle|=\cO(\varepsilon),\\
& \|[d(x^{K+1},y^{K+1})]_+\|=\cO(\varepsilon), \quad |\langle \lambda_\bfy^{K+1}, d(x^{K+1},y^{K+1}) \rangle|=\cO(\varepsilon), 
\end{align*}
where $\tl_\bfx^{K+1}\in\bR_+^{\tn}$ is defined in \eqref{tlx} and $\lambda_\bfy^{K+1}\in\bR_+^{\tm}$ is given in Algorithm \ref{AL-alg}.
\ei
\end{rem}

\section{Proof of the main result} \label{sec:proofs}
In this section we provide a proof of our main results presented in Sections \ref{minimax} and \ref{sec:main}, which are particularly Theorems \ref{ea-prop}, \ref{mmax-thm} and \ref{complexity}.

\subsection{Proof of the main results in Subsection~\ref{strong-cvx-ccv}}\label{sec:proof2}
In this subsection we prove Theorem \ref{ea-prop}. Before proceeding, we establish an upper bound on $\vartheta_0$ in terms of the function value gap of \eqref{ea-prob}, where $\vartheta_0$ is given in \eqref{ea-L}.

\begin{lemma} \label{lem:ea-L0}
Suppose that Assumptions~\ref{mmax-a} and~\ref{ea} hold. Let $\H^*$, $\H_{\rm low}$, $\vartheta_0$ and $\bar \delta$ be defined in \eqref{ea-prob}, \eqref{ea-bnd}, \eqref{ea-L} and \eqref{ea-tP}, and $\bar\alpha$ be given in Algorithm~\ref{mmax-alg1}. Then we have
\begin{equation}\label{ea-L0}
\vartheta_0\leq \bar \delta+2\bar \alpha^{-1}\left(\H^*-\H_{\rm low}\right).
\end{equation}
\end{lemma}

\begin{proof}
By \eqref{ea-prob}, \eqref{ea-bnd}, \eqref{ea-hatf} and \eqref{ea-G}, one has
\begin{align}
\cG(\bar z^0,\bar y^0)\overset{\eqref{ea-G}}{=}&\ \sup_{x}\left\{\langle x,\bar z^0\rangle-p(x)-\hat h(x,\bar y^0)+q(\bar y^0)\right\}\notag\\
\overset{\eqref{ea-hatf}}{=}&\ \max_{x\in\dom\,p}\left\{\langle x,\bar z^0\rangle-p(x)- \h(x,\bar y^0)+\frac{\sigma_x}{2}\|x\|^2-\frac{\sigma_y}{2}\|\bar y^0\|^2+q(\bar y^0)\right\}\notag\\
\overset{\eqref{ea-prob}\eqref{ea-bnd}}{\leq}&\ \max_{x\in\dom\,p}\left\{\langle x,\bar z^0\rangle+\frac{\sigma_x}{2}\|x\|^2\right\}-\frac{\sigma_y}{2}\|\bar y^0\|^2-\H_{\rm low}\notag\\
=\ &\ \max_{x\in\dom\,p}\frac{\sigma_x}{2}\|x+\sigma_x^{-1}\bar z^0\|^2-\frac{\sigma_x^{-1}}{2}\|\bar z^0\|^2-\frac{\sigma_y}{2}\|\bar y^0\|^2-\H_{\rm low}\notag\\
\leq\ &\ \frac{\sigma_xD_\bfx^2}{2}-\frac{\sigma_x^{-1}}{2}\|\bar z^0\|^2-\frac{\sigma_y}{2}\|\bar y^0\|^2-\H_{\rm low},\label{ea-l1}
\end{align}
where the last inequality follows from \eqref{mmax-D}, $\mcX=\dom\,p$, and the fact that $z^0\in-\sigma_x\dom\,p$.

Recall that $(x^*,y^*)$ is the optimal solution of \eqref{ea-prob} and $z^*=-\sigma_x x^*$. It follows from \eqref{ea-prob}, \eqref{ea-hatf} and \eqref{ea-G} that 
\begin{align*}
\cG(z^*,y^*)\overset{\eqref{ea-G}}{=}&\ \sup_{x}\left\{\langle x,z^*\rangle-p(x)-\hat h(x,y^*)+q(y^*)\right\}\geq\langle x^*,z^*\rangle-p(x^*)-\hat h(x^*,y^*)+q(y^*)\\
\overset{\eqref{ea-hatf}}{=}&\ \langle x^*,z^*\rangle+\frac{\sigma_x}{2}\|x^*\|^2-\frac{\sigma_y}{2}\|y^*\|^2-p(x^*)-\h(x^*,y^*)+q(y^*)\\
=\ &\ -\frac{\sigma_x^{-1}}{2}\|z^*\|^2-\frac{\sigma_y}{2}\|y^*\|^2-\H^*,
\end{align*}
where the last equality follows from \eqref{ea-prob}, the definition of $(x^*,y^*)$, and $z^*=-\sigma_xx^*$. This together with \eqref{ea-P} and \eqref{ea-l1} implies that
\begin{align*}
\cP(\bar z^0,\bar y^0)-\cP(z^*,y^*)=&\ \frac{\sigma_x^{-1}}{2}\|\bar z^0\|^2+\frac{\sigma_y}{2}\|\bar y^0\|^2+\cG(\bar z^0,\bar y^0)-\frac{\sigma_x^{-1}}{2}\|z^*\|^2-\frac{\sigma_y}{2}\|y^*\|^2-\cG(z^*,y^*)\\
\leq &\ \sigma_xD_\bfx^2/2-\H_{\rm low}+\H^*.
\end{align*}
Notice from Algorithm~\ref{mmax-alg1} that $z^0=z^0_f=\bar z^0\in-\sigma_x\dom\,p$ and $y^0=y^0_f=\bar y^0\in\dom\,q$. By these, $z^*=-\sigma_x x^*$, $\mcX=\dom\,p$, $\mcY=\dom\,q$, \eqref{mmax-D}, \eqref{ea-L}, and the above inequality, one has
\begin{align*}
\vartheta_0\overset{\eqref{ea-L}}=&\ \eta_z^{-1}\|\bar z^0-z^*\|^2+\eta_y^{-1}\|\bar y^0-y^*\|^2+2\bar \alpha^{-1}(\cP(\bar z^0,\bar y^0)-\cP(z^*,y^*))\\
\leq\ &\ \eta_z^{-1}\sigma_x^2D_\bfx^2+\eta_y^{-1}D_\bfy^2+2\bar \alpha^{-1}\left(\sigma_xD_\bfx^2/2-\H_{\rm low}+\H^*\right)\\
=\ &\ \eta_z^{-1}\sigma_x^2D_\bfx^2+\bar \alpha^{-1}\sigma_xD_\bfx^2+\eta_y^{-1}D_\bfy^2+2\bar \alpha^{-1}\left(\H^*-\H_{\rm low}\right).
\end{align*}
Hence, the conclusion follows from this, \eqref{ea-tP}, $\eta_z=\sigma_x/2$ and $\eta_y=\min\left\{1/(2\sigma_y),4/(\bar\alpha\sigma_x)\right\}$.
\end{proof}

We are now ready to prove Theorem~\ref{ea-prop}, using Lemma \ref{lem:ea-L0}, 
 \cite[Theorem~3]{kovalev2022first},  \cite[Lemma~4]{kovalev2022first}, and \cite[Corollary~2.5]{chen1997convergence}.

\begin{proof}[\textbf{Proof of Theorem~\ref{ea-prop}}]
Suppose for contradiction that Algorithm~\ref{mmax-alg1} runs for more than $\bar K$ outer iterations, where $\bar K$ is given in \eqref{ea-K}. By this and Algorithm~\ref{mmax-alg1}, one can assert that \eqref{ea-term} does not hold for $k=\bar K-1$. On the other hand, by \eqref{ea-K} and \cite[Theorem~3]{kovalev2022first}, one has
\begin{equation}\label{xk-xstar}
\|(x^{\bar K},y^{\bar K})-(x^*,y^*)\|\leq (\bar\zeta^{-1}+L_{\nabla \h})^{-1}\bar\epsilon/2,
\end{equation}
where $(x^*,y^*)$ is the optimal solution of problem \eqref{ea-prob} and $\bar\zeta$ is an input of Algorithm~\ref{mmax-alg1}. 
Notice from Algorithm~\ref{mmax-alg1} that $(\tx^{\bar K},\ty^{\bar K})$ results from the forward-backward splitting (FBS) step applied to the strongly monotone inclusion problem $0\in (\nabla_x \h(x,y),-\nabla_y \h(x,y))+(\partial p(x),\partial q(y))$ at the point $(x^{\bar K},y^{\bar K})$. It then follows from this, $\bar\zeta=\min\{\sigma_x,\sigma_y\}/L_{\nabla \h}^2$ (see Algorithm~\ref{mmax-alg1}), and 
the contraction property of FBS \cite[Corollary~2.5]{chen1997convergence} that $\|(\tx^{\bar K},\ty^{\bar K})-(x^*,y^*)\|\leq\|(x^{\bar K},y^{\bar K})-(x^*,y^*)\|$. Using this and \eqref{xk-xstar}, we have
\begin{align*}
&\|\bar\zeta^{-1}(x^{\bar K}-\tx^{\bar K},\ty^{\bar K}-y^{\bar K})-(\nabla \h(x^{\bar K},y^{\bar K})-\nabla \h(\tx^{\bar K},\ty^{\bar K}))\|\\
&\leq\ \bar\zeta^{-1}\|(x^{\bar K},y^{\bar K})-(\tx^{\bar K},\ty^{\bar K})\|+\|\nabla \h(x^{\bar K},y^{\bar K})-\nabla \h(\tx^{\bar K},\ty^{\bar K})\|\\
&\leq\ (\bar\zeta^{-1}+ L_{\nabla \h})\|(x^{\bar K},y^{\bar K})-(\tx^{\bar K},\ty^{\bar K})\|\\
&\leq\ (\bar\zeta^{-1}+ L_{\nabla \h})(\|(x^{\bar K},y^{\bar K})-(x^{*},y^{*})\|+\|(\tx^{\bar K},\ty^{\bar K})-(x^{*},y^{*})\|)\\
&\leq\ 2(\bar\zeta^{-1}+ L_{\nabla \h})\|(x^{\bar K},y^{\bar K})-(x^*,y^*)\|\overset{\eqref{xk-xstar}}{\leq}\bar\epsilon,
\end{align*}
where the second inequality uses the fact that $\h$ is $L_{\nabla \h}$-smooth on $\dom\,p\times\dom\,q$. It follows that \eqref{ea-term} holds for $k=\bar K-1$, which contradicts the above assertion. Hence, Algorithm~\ref{mmax-alg1} must terminate in at most $\bar K$ outer iterations.

We next show that the output of Algorithm~\ref{mmax-alg1} is an $\bar\epsilon$-primal-dual stationary point of \eqref{ea-prob}. To this end, suppose that Algorithm~\ref{mmax-alg1} terminates at some iteration $k$ at which \eqref{ea-term} is satisfied. Then by \eqref{eq:def-prox} and the definition of $\tx^{k+1}$ and $\ty^{k+1}$ (see steps 23 and 24 of Algorithm~\ref{mmax-alg1}), one has
\begin{eqnarray*}
&0\in\bar\zeta\partial p(\tx^{k+1})+\tx^{k+1}-x^{k+1}+\bar\zeta\nabla_x\h(x^{k+1},y^{k+1}),\\
&0\in\bar\zeta\partial q(\ty^{k+1})+\ty^{k+1}-y^{k+1}-\bar\zeta\nabla_y\h(x^{k+1},y^{k+1}),
\end{eqnarray*}
which yield 
\begin{align*}
\bar\zeta^{-1}(x^{k+1}-\tx^{k+1})-\nabla_x \h(x^{k+1},y^{k+1})\in\partial p(\tx^{k+1}),\ \bar\zeta^{-1}(y^{k+1}-\ty^{k+1})+\nabla_y \h(x^{k+1},y^{k+1})\in\partial q(\ty^{k+1}).
\end{align*}
These together with the definition of $\H$ in \eqref{ea-prob} imply that
\begin{align*}
&\nabla_x \h(\tx^{k+1},\ty^{k+1})+\bar\zeta^{-1}(x^{k+1}-\tx^{k+1})-\nabla_x \h(x^{k+1},y^{k+1})\in\partial_x \H(\tx^{k+1},\ty^{k+1}),\\
&\nabla_y \h(\tx^{k+1},\ty^{k+1})-\bar\zeta^{-1}(y^{k+1}-\ty^{k+1})-\nabla_y \h(x^{k+1},y^{k+1})\in\partial_y \H(\tx^{k+1},\ty^{k+1}).
\end{align*}
Using these and \eqref{ea-term}, we obtain
\begin{align*}
& \dist(0,\partial_x \H(\tx^{k+1},\ty^{k+1}))^2+\dist(0,\partial_y \H(\tx^{k+1},\ty^{k+1}))^2\\
&\leq \|\bar\zeta^{-1}(x^{k+1}-\tx^{k+1})+\nabla_x \h(\tx^{k+1},\ty^{k+1})-\nabla_x \h(x^{k+1},y^{k+1})\|^2\\
& \quad +\|\bar\zeta^{-1}(\ty^{k+1}-y^{k+1})+\nabla_y \h(\tx^{k+1},\ty^{k+1})-\nabla_y \h(x^{k+1},y^{k+1})\|^2\\
& = \|\bar\zeta^{-1}(x^{k+1}-\tx^{k+1},\ty^{k+1}-y^{k+1})-(\nabla \h(x^{k+1},y^{k+1})-\nabla \h(\tx^{k+1},\ty^{k+1}))\|^2\overset{\eqref{ea-term}}{\leq}\bar\epsilon^2,
\end{align*}
which implies that $\dist(0,\partial_x \H(\tx^{k+1},\ty^{k+1})) \leq \bar\epsilon$ and $\dist(0,\partial_y \H(\tx^{k+1},\ty^{k+1}))\leq \bar\epsilon$. It then follows from these and Definition \ref{def2} that the output $(\tx^{k+1},\ty^{k+1})$ of Algorithm~\ref{mmax-alg1} is an $\bar\epsilon$-primal-dual stationary point of \eqref{ea-prob}.

Finally, we show that the total number of evaluations of $\nabla \h$ and proximal operator of $p$ and $q$ performed in Algorithm~\ref{mmax-alg1} is no more than $\bar N$, respectively. Indeed, notice from Algorithm~\ref{mmax-alg1} that  $\bar \alpha=\min\left\{1,\sqrt{8\sigma_y/\sigma_x}\right\}$, which implies that 
$2/\bar \alpha = \max\{2,\sqrt{\sigma_x/(2\sigma_y)}\}$ and $\bar \alpha \leq \sqrt{8\sigma_y/\sigma_x}$. By these, one has
\begin{equation}\label{alpha-bnd}
\max\left\{\frac{2}{\bar \alpha},\frac{\bar \alpha\sigma_x}{4\sigma_y}\right\} \leq 
\max\left\{2,\sqrt{\frac{\sigma_x}{2\sigma_y}},\sqrt{\frac{8\sigma_y}{\sigma_x}}\frac{\sigma_x}{4\sigma_y}\right\}
 =\max\left\{2,\sqrt{\frac{\sigma_x}{2\sigma_y}}\right\}.
\end{equation}
In addition, by \cite[Lemma~4]{kovalev2022first}, the number of inner iterations performed in each outer iteration of Algorithm~\ref{mmax-alg1} is at most
\begin{equation*}
\bar T:=\left\lceil48\sqrt{2}\left(1+8L_{\nabla \h}\sigma_x^{-1}\right)\right\rceil-1.
\end{equation*}
Then one can observe that the number of evaluations of $\nabla \h$ and proximal operator of $p$ and $q$ performed in Algorithm~\ref{mmax-alg1} is at most 
\begin{align*}
& (2\bar T+3) \bar K \leq \left(\left\lceil96\sqrt{2}\left(1+8L_{\nabla \h}\sigma_x^{-1}\right)\right\rceil+2\right)\left\lceil\max\left\{\frac{2}{\bar \alpha},\frac{\bar \alpha\sigma_x}{4\sigma_y}\right\}\log\frac{4\max\{\eta_z\sigma_x^{-2},\eta_y\}\vartheta_0}{(\bar\zeta^{-1}+ L_{\nabla \h})^{-2}\bar\epsilon^2}\right\rceil_+\\
& \overset{\eqref{alpha-bnd}}{\leq} \left(\left\lceil96\sqrt{2}\left(1+8L_{\nabla \h}\sigma_x^{-1}\right)\right\rceil+2\right)
\left\lceil\max\left\{2,\sqrt{\frac{\sigma_x}{2\sigma_y}}\right\}\log\frac{4\max\{\eta_z\sigma_x^{-2},\eta_y\}\vartheta_0}{(\bar\zeta^{-1}+ L_{\nabla \h})^{-2}\bar\epsilon^2}\right\rceil_+\\
&\ \leq\left(\left\lceil96\sqrt{2}\left(1+8L_{\nabla \h}\sigma_x^{-1}\right)\right\rceil+2\right)\\
&\ \ \ \ \times\left\lceil\max\left\{2,\sqrt{\frac{\sigma_x}{2\sigma_y}}\right\}\log\frac{4\max\{1/(2\sigma_x),\min\left\{1/(2\sigma_y),4/(\bar \alpha\sigma_x)\}\right\}\vartheta_0}{(L_{\nabla \h}^2/\min\{\sigma_x,\sigma_y\}+ L_{\nabla \h})^{-2}\bar\epsilon^2}\right\rceil_+\overset{\eqref{ea-N}\eqref{ea-L0}}{\leq}\bar N,
\end{align*}
where the second last inequality follows from the definition of $\eta_y$, $\eta_z$ and $\bar\zeta$ in Algorithm~\ref{mmax-alg1}. Hence, the conclusion holds as desired.
\end{proof}

\subsection{Proof of the main results in Subsection~\ref{ppa}}\label{sec:proof2-2}
In this subsection we prove Theorem \ref{mmax-thm}. Before proceeding, let $\{(x^k,y^k)\}_{k\in \bbT}$ denote all the iterates generated by Algorithm~\ref{mmax-alg2}, where $\bbT$ is a subset of consecutive nonnegative integers starting from $0$. Also, we define $\bbT-1 = \{k-1: k \in \bbT\}$. We first establish two lemmas and then use them to prove Theorem \ref{mmax-thm} subsequently.

The following lemma shows that an approximate primal-dual stationary point  of \eqref{ppa-subprob} is found at each iteration of Algorithm~\ref{mmax-alg2}, and also provides an estimate of operation complexity for finding it.

\begin{lemma}\label{innercplx-alg6}
Suppose that Assumption~\ref{mmax-a} holds. Let $\{(x^k,y^k)\}_{k\in\bbT}$ be generated by Algorithm~\ref{mmax-alg2}, $\bH^*$, $D_\bfx$, $D_\bfy$, $\bH_{\rm low}$, $\halpha$, $\hdelta$ be defined in \eqref{mmax-prob}, \eqref{mmax-D}, \eqref{mmax-bnd}, \eqref{mmax-balpha} and \eqref{mmax-tP}, $L_{\nabla\bh}$ be given in Assumption \ref{mmax-a}, $\epsilon$, $\hat\epsilon_k$ be given in Algorithm~\ref{mmax-alg2}, and
\begin{align}
&\ \hat N_k:=\left(\left\lceil96\sqrt{2}\left(1+\left(24L_{\nabla \bh}+4\epsilon/D_\bfy\right)L_{\nabla \bh}^{-1}\right)\right\rceil+2\right)\times\Bigg\lceil\max\left\{2,\sqrt{\frac{D_\bfy L_{\nabla \bh}}{\epsilon}}\right\}\notag\\
&\ \ \ \ \times\log\frac{4\max\left\{\frac{1}{2L_{\nabla \bh}},\min\left\{\frac{D_\bfy}{\epsilon},\frac{4}{\halpha L_{\nabla \bh}}\right\}\right\}\left(\hdelta+2\halpha^{-1}(\bH^*-\bH_{\rm low}+\epsilon D_\bfy/4+L_{\nabla \bh} D_\bfx^2)\right)}{\left[(3L_{\nabla \bh}+\epsilon/(2D_\bfy))^2/\min\{L_{\nabla \bh},\epsilon/(2D_\bfy)\}+ 3L_{\nabla \bh}+\epsilon/(2D_\bfy)\right]^{-2}\hat\epsilon_k^2}\Bigg\rceil_+.\label{mmax-Nk}
\end{align}
Then for all $0\leq k\in\bbT-1$, $(x^{k+1},y^{k+1})$ is an $\hat\epsilon_k$-primal-dual stationary point of \eqref{ppa-subprob}. Moreover, the total number of evaluations of $\nabla \bh$ and proximal operator of $p$ and $q$ performed at iteration $k$ of Algorithm~\ref{mmax-alg2} for generating $(x^{k+1},y^{k+1})$ is no more than $\hat N_k$, respectively. 
\end{lemma}

\begin{proof}
Let $(x^*,y^*)$ be an optimal solution of \eqref{mmax-prob}. Recall that $H$, $H_k$ and $h_k$ are respectively given in \eqref{mmax-prob}, \eqref{ppa-subprob} and \eqref{mmax-sub}, $\mcX=\dom\,p$ and $\mcY=\dom\,q$. Notice that $x^*, x^k\in\mcX$. Then we have
\begin{align}
\bH_{k,*}:=\min_x\max_y H_k(x,y) =&\ \min_x\max_y\left \{\bH(x,y)-\frac{\epsilon}{4D_\bfy}\|y-\hat y^0\|^2+L_{\nabla \bh}\|x-x^k\|^2\right\}\notag\\
\leq&\ \max_y \{\bH(x^*,y)+L_{\nabla \bh}\|x^*-x^k\|^2\} \overset{\eqref{mmax-prob}\eqref{mmax-D}}\leq \bH^*+L_{\nabla \bh} D_\bfx^2.\label{Hkstar}
\end{align}
 Moreover, by $\mcX=\dom\,p$, $\mcY=\dom\,q$, \eqref{mmax-D} and \eqref{mmax-bnd}, one has
\begin{align}
\bH_{k, \rm low}:=\min_{(x,y)\in\dom\,p\times\dom\,q} H_k(x,y) =& \ \min_{(x,y)\in\mcX\times\mcY}\left\{\bH(x,y)-\frac{\epsilon}{4D_\bfy}\|y-\hat y^0\|^2+L_{\nabla \bh}\|x-x^k\|^2\right\}\notag\\
\overset{\eqref{mmax-bnd}}\geq
& \ \bH_{\rm low}-\max_{y\in\mcY}\frac{\epsilon}{4D_\bfy}\|y-\hat y^0\|^2\overset{\eqref{mmax-D}}{\geq} \bH_{\rm low}-\epsilon D_\bfy/4.\label{hklow}
\end{align}
In addition, by Assumption~\ref{mmax-a} and the definition of $\bh_k$ in \eqref{mmax-sub}, it is not hard to verify that $\bh_k(x,y)$ is $L_{\nabla \bh}$-strongly-convex in $x$, $\epsilon/(2D_\bfy)$-strongly-concave in $y$, and $(3L_{\nabla \bh}+\epsilon/(2D_\bfy))$-smooth on its domain. Also, recall from Remark \ref{ppa-rem} that $(x^{k+1},y^{k+1})$  results from applying Algorithm~\ref{mmax-alg1} to problem \eqref{ppa-subprob}.  
The conclusion of this lemma then follows by using \eqref{Hkstar} and \eqref{hklow} and applying Theorem~\ref{ea-prop} to  \eqref{ppa-subprob} with $\bar\epsilon=\hat\epsilon_k$, $\sigma_x=L_{\nabla \bh}$, $\sigma_y=\epsilon/(2D_\bfy)$, $L_{\nabla \h}=3L_{\nabla \bh}+\epsilon/(2D_\bfy)$, $\bar \alpha= \halpha$, $\bar \delta=\hdelta$, $\H_{\rm low}=H_{k, \rm low}$, and $\H^*=H_{k,*}$.
\end{proof}

The following lemma provides an upper bound on the least progress of the solution sequence of Algorithm~\ref{mmax-alg2} and also on the last-iterate objective value of \eqref{mmax-prob}.

\begin{lemma} \label{output-prop}
Suppose that Assumption~\ref{mmax-a} holds. Let $\{x^k\}_{k\in\bbT}$ be generated by Algorithm~\ref{mmax-alg2}, $H$, $\bH^*$ and $D_\bfy$ be defined in \eqref{mmax-prob} and \eqref{mmax-D}, $L_{\nabla\bh}$ be given in Assumption \ref{mmax-a}, and $\epsilon$, $\hat\epsilon_0$ and $\hat x^0$ be given in Algorithm~\ref{mmax-alg2}. Then for all $0\leq K\in\bbT-1$, we have
\begin{align}
&\min_{0\leq k\leq K}\|x^{k+1}-x^k\|\leq\frac{\max_y\bH(\hat x^0,y)-\bH^*+\epsilon D_\bfy/4}{L_{\nabla \bh}(K+1)}+\frac{2\hat\epsilon_0^{2}(1+4D_\bfy^2L_{\nabla \bh}^2\epsilon^{-2})}{L_{\nabla \bh}^2(K+1)},\label{mmax-xbnd}\\
&\max_y\bH(x^{K+1},y)\leq \max_y\bH(\hat x^0,y)+\epsilon D_\bfy/4+2\hat\epsilon_0^2\left(L_{\nabla \bh}^{-1}+4D_\bfy^2L_{\nabla \bh}\epsilon^{-2}\right).\label{K-upperbnd}
\end{align}
\end{lemma}

\begin{proof}
For convenience of the proof, let
\begin{align}
&\hh(x)=\max\limits_y\left\{\bH(x,y)-\epsilon\|y-\hat y^0\|^2/(4D_\bfy)\right\},\label{mmax-tg}\\
&\bH^*_k(x)=\max\limits_y \bH_k(x,y),\quad y^{k+1}_*=\argmax_y\bH_k(x^{k+1},y). \label{mmax-ystar}
\end{align}
One can observe from these, \eqref{ppa-subprob} and \eqref{mmax-sub} that
\beq \label{mmax-gk}
H^*_k(x)=\hh(x)+L_{\nabla \bh}\|x-x^k\|^2.  
\eeq
By this and Assumption~\ref{mmax-a}, one can also see that $H^*_k$ is $L_{\nabla \bh}$-strongly convex on $\dom\,p$.
In addition, recall from Lemma \ref{innercplx-alg6} that $(x^{k+1},y^{k+1})$ is an $\hat\epsilon_k$-primal-dual stationary point of problem \eqref{ppa-subprob} for all $0\leq k\in\bbT-1$. It then follows from Definition~\ref{def2} that there exist some $u\in\partial_x \bH_k(x^{k+1},y^{k+1})$ and $v\in\partial_y \bH_k(x^{k+1},y^{k+1})$ with $\|u\|\leq\hat\epsilon_k$ and $\|v\|\leq\hat\epsilon_k$. Also, by \eqref{mmax-ystar}, one has $0\in\partial_y \bH_k(x^{k+1},y^{k+1}_*)$, which together with $v\in\partial_y \bH_k(x^{k+1},y^{k+1})$ and  $\epsilon/(2D_\bfy)$-strong concavity of $\bH_k(x^{k+1},\cdot)$, implies that
$\langle -v,y^{k+1}-y^{k+1}_*\rangle\geq \epsilon\|y^{k+1}-y^{k+1}_*\|^2/(2D_\bfy)$. 
This and $\|v\|\leq\hat\epsilon_k$ yield  
\begin{equation}\label{mmax-e1}
\|y^{k+1}-y^{k+1}_*\|\leq2\hat\epsilon_kD_\bfy/\epsilon.
\end{equation}
In addition, by $u\in\partial_x \bH_k(x^{k+1},y^{k+1})$, \eqref{ppa-subprob} and \eqref{mmax-sub}, one has \begin{equation}\label{mmax-e2}
u\in\nabla_x\bh(x^{k+1},y^{k+1})+\partial p(x^{k+1})+2L_{\nabla \bh}(x^{k+1}-x^k).
\end{equation}
Also, observe from \eqref{ppa-subprob}, \eqref{mmax-sub} and \eqref{mmax-ystar} that
\begin{equation*}
\partial H^*_k(x^{k+1})=\nabla_x\bh(x^{k+1},y^{k+1}_*)+\partial p(x^{k+1})+2L_{\nabla \bh}(x^{k+1}-x^k),
\end{equation*}
which together with \eqref{mmax-e2} yields
\begin{equation*}
u+\nabla_x\bh(x^{k+1},y^{k+1}_*)-\nabla_x\bh(x^{k+1},y^{k+1})\in\partial H^*_k(x^{k+1}).
\end{equation*}
By this and $L_{\nabla \bh}$-strong convexity of $H^*_k$, one has
\begin{equation}\label{strongcvx}
H^*_k(x^k)\geq H^*_k(x^{k+1})+\langle u+\nabla_x\bh(x^{k+1},y^{k+1}_*)-\nabla_x\bh(x^{k+1},y^{k+1}),x^k-x^{k+1}\rangle+L_{\nabla \bh}\|x^k-x^{k+1}\|^2/2.
\end{equation}
Using this, \eqref{mmax-gk}, \eqref{mmax-e1}, \eqref{strongcvx}, $\|u\|\leq\hat\epsilon_k$, and the Lipschitz continuity of $\nabla \bh$, we obtain
\begin{align*}
&\hh(x^k)-\hh(x^{k+1})\overset{\eqref{mmax-gk}}{=}H^*_k(x^k)- H^*_k(x^{k+1})+L_{\nabla\bh}\|x^k-x^{k+1}\|^2\\
&\overset{\eqref{strongcvx}}{\geq}\langle u+\nabla_x\bh(x^{k+1},y^{k+1}_*)-\nabla_x\bh(x^{k+1},y^{k+1}),x^k-x^{k+1}\rangle+3L_{\nabla \bh}\|x^k-x^{k+1}\|^2/2\\
&\ \geq \big(-\|u+\nabla_x\bh(x^{k+1},y^{k+1}_*)-\nabla_x\bh(x^{k+1},y^{k+1})\|\|x^k-x^{k+1}\|+L_{\nabla \bh}\|x^k-x^{k+1}\|^2/2\big) +L_{\nabla \bh}\|x^k-x^{k+1}\|^2 \\
&\ \geq -(2L_{\nabla \bh})^{-1}\|u+\nabla_x\bh(x^{k+1},y^{k+1}_*)-\nabla_x\bh(x^{k+1},y^{k+1})\|^2+L_{\nabla \bh}\|x^k-x^{k+1}\|^2 \\
&\ \geq-L_{\nabla \bh}^{-1}\|u\|^2-L_{\nabla \bh}^{-1}\|\nabla_x \bh(x^{k+1},y^{k+1}_*)-\nabla_x \bh(x^{k+1},y^{k+1})\|^2+L_{\nabla \bh}\|x^k-x^{k+1}\|^2\\
&\ \geq-L_{\nabla \bh}^{-1}\hat\epsilon_k^2-L_{\nabla\bh}\|y^{k+1}-y^{k+1}_*\|^2+L_{\nabla\bh}\|x^k-x^{k+1}\|^2\\
&\overset{\eqref{mmax-e1}}{\geq} -(L_{\nabla\bh}^{-1}+4D_\bfy^2L_{\nabla\bh}\epsilon^{-2})\hat\epsilon_k^2+L_{\nabla\bh}\|x^k-x^{k+1}\|^2, 
\end{align*}
where the second and fourth inequalities follow from Cauchy-Schwartz inequality, and the third inequality is due to Young's inequality, and the fifth inequality follows from $L_{\nabla \bh}$-Lipschitz continuity of $\nabla \bh$. Summing up the above inequality for $k=0,1,\dots, K$ yields
\begin{equation}\label{mmax-sum}
L_{\nabla \bh}\sum_{k=0}^K\|x^k-x^{k+1}\|^2\leq\hh(x^0)-\hh(x^{K+1})+(L_{\nabla\bh}^{-1}+4D_\bfy^2L_{\nabla \bh}\epsilon^{-2})\sum_{k=0}^K\hat\epsilon_k^2.
\end{equation}
In addition, it follows from \eqref{mmax-prob}, \eqref{mmax-D} and \eqref{mmax-tg} that
\begin{align}
&\hh(x^{K+1})=\max_y\left\{\bH(x^{K+1},y)-\epsilon\|y-\hat y^0\|^2/(4D_\bfy)\right\}\geq\min_x\max_y\bH(x,y)-\epsilon D_\bfy/4=\bH^*-\epsilon D_\bfy/4,\notag\\
&\hh(x^0)=\max_y\left\{\bH(x^0,y)-\epsilon\|y-\hat y^0\|^2/(4D_\bfy)\right\}\leq\max_y\bH(x^0,y).\label{hhx0}
\end{align}
These together with \eqref{mmax-sum} yield
\begin{align*}
L_{\nabla \bh}(K+1)\min_{0\leq k\leq K}\|x^{k+1}-x^k\|^2\leq&\ L_{\nabla \bh}\sum_{k=0}^K\|x^k-x^{k+1}\|^2\\
\leq&\ \max_y\bH(x^0,y)-\bH^*+\epsilon D_\bfy/4+(L_{\nabla\bh}^{-1}+4D_\bfy^2L_{\nabla \bh}\epsilon^{-2})\sum_{k=0}^K\hat\epsilon_k^2,
\end{align*}
which together with $x^0=\hat x^0$, $\hat\epsilon_k=\hat\epsilon_0(k+1)^{-1}$ and $\sum_{k=0}^K(k+1)^{-2}<2$ implies that \eqref{mmax-xbnd} holds.

Finally, we show that \eqref{K-upperbnd} holds. Indeed, it follows from \eqref{mmax-D}, \eqref{mmax-tg}, \eqref{mmax-sum}, \eqref{hhx0}, $\hat\epsilon_k=\hat\epsilon_0(k+1)^{-1}$, and $\sum_{k=0}^K(k+1)^{-2}<2$ that
\begin{align*}
\max_y\bH(x^{K+1},y) \overset{\eqref{mmax-D}}{\leq} \ &\ \max_{y}\left\{\bH(x^{K+1},y)-\epsilon\|y-\hat y^0\|^2/(4D_\bfy)\right\}+\epsilon D_\bfy/4\overset{\eqref{mmax-tg}}{=}\hh(x^{K+1})+\epsilon D_\bfy/4 \\
\overset{\eqref{mmax-sum}}{\leq}&\ \hh(x^0)+\epsilon D_\bfy/4+(L_{\nabla\bh}^{-1}+4D_\bfy^2L_{\nabla \bh}\epsilon^{-2})\sum_{k=0}^K\hat\epsilon_k^2\\
\overset{\eqref{hhx0}}{\leq}&\ \max_y\bH(x^0,y)+\epsilon D_\bfy/4+2\hat\epsilon_0^2(L_{\nabla\bh}^{-1}+4D_\bfy^2L_{\nabla \bh}\epsilon^{-2}).
\end{align*}
It then follows from this and $x^0=\hat x^0$ that \eqref{K-upperbnd} holds.
\end{proof}

We are now ready to prove Theorem~\ref{mmax-thm} using Lemmas \ref{innercplx-alg6} and \ref{output-prop}.

\begin{proof}[\textbf{Proof of Theorem~\ref{mmax-thm}}]
Suppose for contradiction that Algorithm~\ref{mmax-alg2} runs for more than $\wT+1$ outer iterations, where $\wT$ is given in \eqref{mmax-K}. By this and Algorithm~\ref{mmax-alg2}, one can then assert that \eqref{mmax-term} does not hold for all $0\leq k\leq T$. On the other hand, by \eqref{mmax-K} and \eqref{mmax-xbnd}, one has
\begin{align*}
\min_{0\leq k\leq \wT}\|x^{k+1}-x^k\|^2\ \overset{\eqref{mmax-xbnd}}{\leq}\ \frac{\max_y\bH(\hat x^0,y)-\bH^*+\epsilon D_\bfy/4}{L_{\nabla \bh}(\wT+1)}+\frac{2\hat\epsilon_0^{2}(1+4D_\bfy^2L_{\nabla \bh}^2\epsilon^{-2})}{L_{\nabla \bh}^2(\wT+1)}\overset{\eqref{mmax-K}}{\leq}\frac{\epsilon^2}{16L_{\nabla \bh}^2},
\end{align*}
which implies that there exists some $0\leq k\leq \wT$ such that $\|x^{k+1}-x^k\|\leq\epsilon/(4L_{\nabla \bh})$,
and thus \eqref{mmax-term} holds for such $k$, which contradicts the above assertion. Hence, Algorithm~\ref{mmax-alg2} must terminate in at most $\wT+1$ outer iterations.

Suppose that Algorithm~\ref{mmax-alg2} terminates at some iteration $0\leq k\leq \wT$, namely,   \eqref{mmax-term} holds for such $k$. We next show that its output $(\xe,\ye)=(x^{k+1},y^{k+1})$ is an $\epsilon$-primal-dual stationary point of \eqref{mmax-prob} and moreover it satisfies 
\eqref{upperbnd}. Indeed, recall from Lemma \ref{innercplx-alg6} that $(x^{k+1},y^{k+1})$ is an $\hat\epsilon_k$-primal-dual stationary point of 
\eqref{ppa-subprob}, namely, it satisfies $\dist(0,\partial_x \bH_k(x^{k+1},y^{k+1})) \leq \hat\epsilon_k$ and $\dist(0,\partial_y \bH_k(x^{k+1},y^{k+1})) \leq \hat\epsilon_k$. By these, \eqref{mmax-prob}, \eqref{ppa-subprob} and \eqref{mmax-sub},  there exists $(u,v)$ such that 
\begin{align*}
&u\in\partial_{x}\bH(x^{k+1},y^{k+1})+2L_{\nabla \bh}(x^{k+1}-x^k),\quad\|u\|\leq\hat\epsilon_k,\\
&v\in\partial_{y}\bH(x^{k+1},y^{k+1})-\epsilon(y^{k+1}-\hat y^0)/(2D_\bfy),\quad\|v\|\leq\hat\epsilon_k.
\end{align*}
It then follows that $u-2L_{\nabla \bh}(x^{k+1}-x^k)\in\partial_{x}\bH(x^{k+1},y^{k+1})$ and $v+\epsilon(y^{k+1}-\hat y^0)/(2D_\bfy)\in\partial_{y}\bH(x^{k+1},y^{k+1})$. These together with \eqref{mmax-D}, \eqref{mmax-term} and $\hat\epsilon_k\leq\hat\epsilon_0\leq\epsilon/2$ (see Algorithm~\ref{mmax-alg2}) imply that
\begin{align*}
&\dist\left(0,\partial_x\bH(x^{k+1},y^{k+1})\right)\leq\|u-2L_{\nabla \bh}(x^{k+1}-x^k)\|\leq\|u\|+2L_{\nabla \bh}\|x^{k+1}-x^k\|\overset{\eqref{mmax-term}}{\leq}\hat\epsilon_k+\epsilon/2\leq\epsilon,\\
&\dist\left(0,\partial_y\bH(x^{k+1},y^{k+1})\right)\leq\|v+\epsilon(y^{k+1}-\hat y^0)/(2D_\bfy)\|\leq\|v\|+\epsilon\|y^{k+1}-\hat y^0\|/(2D_\bfy)\overset{\eqref{mmax-D}}{\leq}\hat\epsilon_k+\epsilon/2\leq\epsilon.
\end{align*}
Hence, the output $(x^{k+1},y^{k+1})$ of Algorithm~\ref{mmax-alg2} is an $\epsilon$-primal-dual stationary point of \eqref{mmax-prob}. In addition, 
\eqref{upperbnd-old} holds due to Lemma \ref{output-prop}.

Recall from Lemma \ref{innercplx-alg6} that the number of evaluations of $\nabla \bh$ and proximal operator of $p$ and $q$ performed at iteration $k$ of Algorithm~\ref{mmax-alg2} is at most $\hat N_k$, respectively, where $\hat N_k$ is defined in \eqref{mmax-Nk}. Also, one can observe from the above proof and the definition of $\bbT$ that $|\bbT|\leq \wT+2$. It then follows that the total number of evaluations of $\nabla \bh$ and proximal operator of $p$ and $q$ in Algorithm~\ref{mmax-alg2} is respectively no more than $\sum_{k=0}^{|\bbT|-2}\hat N_k$. 
Consequently, to complete the rest of the proof of Theorem~\ref{mmax-thm}, it suffices to show that $\sum_{k=0}^{|\bbT|-2}\hat N_k \leq \widehat N$, where $\widehat N$ is given 
in \eqref{mmax-N-old}. Indeed, by \eqref{mmax-N-old}, \eqref{mmax-Nk} and $|\bbT|\leq \wT+2$, one has
\begin{align*}
&\sum_{k=0}^{|\bbT|-2}\hat N_k\overset{\eqref{mmax-Nk}}{\leq}\sum_{k=0}^{\wT}\left(\left\lceil96\sqrt{2}\left(1+\left(24L_{\nabla \bh}+4\epsilon/D_\bfy\right)L_{\nabla \bh}^{-1}\right)\right\rceil+2\right)\times\Bigg\lceil\max\left\{2,\sqrt{\frac{D_\bfy L_{\nabla \bh}}{\epsilon}}\right\}\notag\\
&\ \ \ \ \times\log\frac{4\max\left\{\frac{1}{2L_{\nabla \bh}},\min\left\{\frac{D_\bfy}{\epsilon},\frac{4}{\halpha L_{\nabla \bh}}\right\}\right\}\left( \hdelta+2\halpha^{-1}(\bH^*-\bH_{\rm low}+\epsilon D_\bfy/4+L_{\nabla \bh} D_\bfx^2)\right)}{\left[(3L_{\nabla \bh}+\epsilon/(2D_\bfy))^2/\min\{L_{\nabla \bh},\epsilon/(2D_\bfy)\}+ 3L_{\nabla \bh}+\epsilon/(2D_\bfy)\right]^{-2}\hat\epsilon_k^2}\Bigg\rceil_+\\
&\leq\left(\left\lceil96\sqrt{2}\left(1+\left(24L_{\nabla \bh}+4\epsilon/D_\bfy\right)L_{\nabla \bh}^{-1}\right)\right\rceil+2\right)\max\left\{2,\sqrt{\frac{D_\bfy L_{\nabla \bh}}{\epsilon}}\right\}\notag\\
&\ \ \ \ \times\sum_{k=0}^{\wT}\left(\left(\log\frac{4\max\left\{\frac{1}{2L_{\nabla \bh}},\min\left\{\frac{D_\bfy}{\epsilon},\frac{4}{\halpha L_{\nabla \bh}}\right\}\right\}\left(\hdelta+2\halpha^{-1}(\bH^*-\bh_{\rm low}+\epsilon D_\bfy/4+L_{\nabla \bh} D_\bfx^2)\right)}{\left[(3L_{\nabla \bh}+\epsilon/(2D_\bfy))^2/\min\{L_{\nabla \bh},\epsilon/(2D_\bfy)\}+ 3L_{\nabla \bh}+\epsilon/(2D_\bfy)\right]^{-2}\hat\epsilon_k^2}\right)_++1\right)\\
&\leq\left(\left\lceil96\sqrt{2}\left(1+\left(24L_{\nabla \bh}+4\epsilon/D_\bfy\right)L_{\nabla \bh}^{-1}\right)\right\rceil+2\right)\max\left\{2,\sqrt{\frac{D_\bfy L_{\nabla \bh}}{\epsilon}}\right\}\notag\\
&\ \ \ \ \times\Bigg((\wT+1)\Bigg(\log\frac{4\max\left\{\frac{1}{2L_{\nabla \bh}},\min\left\{\frac{D_\bfy}{\epsilon},\frac{4}{\halpha L_{\nabla \bh}}\right\}\right\}\left(\hdelta+2\halpha^{-1}(\bH^*-\bH_{\rm low}+\epsilon D_\bfy/4+L_{\nabla \bh} D_\bfx^2)\right)}{\left[(3L_{\nabla \bh}+\epsilon/(2D_\bfy))^2/\min\{L_{\nabla \bh},\epsilon/(2D_\bfy)\}+ 3L_{\nabla \bh}+\epsilon/(2D_\bfy)\right]^{-2}\hat\epsilon_0^2}\Bigg)_+\\
&\ \ \ \ +\wT+1+2\sum_{k=0}^{\wT}\log(k+1) \Bigg)\overset{\eqref{mmax-N-old}}{\leq} \widehat N,
\end{align*}
where the last inequality is due to \eqref{mmax-N-old} and $\sum_{k=0}^{\wT}\log(k+1) \leq \wT\log(\wT+1)$. This completes the proof of Theorem~\ref{mmax-thm}.
\end{proof}

\subsection{Proof of the main results in Subsection \ref{sec:complexity}} \label{sec4.3}
In this subsection, we provide a proof of our main result presented in Section \ref{sec:main}, which is particularly Theorem \ref{complexity}. Before proceeding, 
let
\begin{align}
\AL_\bfy(x,y,\lambda_\bfy;\rho)=F(x,y)-\frac{1}{2\rho}\left(\|[\lambda_\bfy+\rho d(x,y)]_+\|^2-\|\lambda_\bfy\|^2\right).\label{y-AL}
\end{align}
In view of \eqref{AL}, \eqref{fstarx} and \eqref{y-AL}, one can observe that
\beq\label{p-ineq}
f^*(x)\leq\max_y\AL_\bfy(x,y,\lambda_\bfy;\rho)\qquad \forall x\in\mcX,\ \lambda_\bfy\in\bR_+^{\tm},\ \rho>0,
\eeq
which will be frequently used later. 

We next establish several lemmas that will be used to prove  Theorem \ref{complexity} subsequently.  The following lemma establishes an upper bound on the optimal Lagrangian multipliers of problem \eqref{fstarx} and also provides a reformulation of $f^*(x)$.

\begin{lemma} \label{dual-bnd}
Suppose that Assumptions \ref{a1} and \ref{mfcq} hold. Let $f^*$, $\Delta$, $r$ and $\delta_d$ be given in \eqref{fstarx}, \eqref{def-r} and Assumption \ref{mfcq}, respectively. Then the following statements hold.
\begin{enumerate}[label=(\roman*)]
\item $\|\lambda^*_\bfy\|\leq \delta_d^{-1}\Delta$ and $\lambda^*_\bfy\in\cB^+_r$ for all $\lambda^*_\bfy\in\Lambda^*(x)$ and $x\in\mcX$, where $\Lambda^*(x)$ denotes the set of optimal Lagrangian multipliers of problem \eqref{fstarx} for any $x\in\mcX$.
\item It holds that
\beq \label{fstar-ref}
f^*(x)=\min_{\lambda_\bfy}\max_y F(x,y)-\langle\lambda_\bfy,d(x,y)\rangle+\delta_{\bR^{\tm}_+}(\lambda_\bfy) \qquad \forall x\in\mcX,
\eeq
where $\delta_{\bR^{\tm}_+}(\cdot)$ is the indicator function associated with $\bR^{\tm}_+$.
\end{enumerate}
\end{lemma}

\begin{proof}
(i) Let $x\in\mcX$, $\lambda^*_\bfy\in\Lambda^*(x)$ be arbitrarily chosen, and $\hat y_x\in\mcY$  and $\delta_d>0$ be given in Assumption~\ref{mfcq}(ii).  It then follows from  Assumption~\ref{mfcq}(ii) that $d_i(x,\hat y_x) \leq -\delta_d$ for all $i$.  In addition, let $y^*\in\mcY$ be such that $(y^*,\lambda^*_\bfy)$ is a pair of primal-dual optimal solutions of \eqref{fstarx}. Then we have
\[
y^*\in\Argmax_{y} F(x,y)-\langle\lambda^*_\bfy,d(x,y)\rangle, \quad \langle\lambda^*_\bfy,d(x,y^*)\rangle=0, \quad d(x,y^*) \leq 0, \quad \lambda^*_\bfy \geq 0.
\]
The first relation above yields
\[
F(x,y^*)-\langle\lambda^*_\bfy,d(x,y^*)\rangle\geq F(x,\hat y_x)-\langle\lambda^*_\bfy,d(x,\hat y_x)\rangle.
\]
 By this and $\langle\lambda^*_\bfy,d(x,y^*)\rangle=0$, one has
\begin{equation*}
\langle\lambda^*_\bfy,-d(x,\hat y_x)\rangle\leq F(x,y^*)-F(x,\hat y_x),
\end{equation*}
which together with $\lambda^*_\bfy\geq 0$, $d_i(x,\hat y_x) \leq -\delta_d$ for all $i$, \eqref{Fhi} and \eqref{def-r} implies that
\[
\delta_d \|\lambda^*_\bfy\|_1 \leq \langle\lambda^*_\bfy,-d(x,\hat y_x)\rangle \leq F(x,y^*)-F(x,\hat y_x) \leq 
\Delta,
\]
Hence, we have $\|\lambda^*_\bfy\|\leq \|\lambda^*_\bfy\|_1 \leq \delta_d^{-1}\Delta$. 
This and  \eqref{def-r} imply that $\lambda^*_\bfy\in\cB^+_r$.

(ii) Recall from Assumption \ref{a1} that $F(x,\cdot)$ and $d_i(x,\cdot), \ i=1,\ldots,l,$  are convex for any given  $x\in\mcX$.  Using this, \eqref{fstarx}, \eqref{def-r} and the first statement of this lemma, we observe that 
\[
f^*(x)=\max_{y}\min_{\lambda\in\cB^+_r}F(x,y)-\langle\lambda,d(x,y)\rangle \qquad \forall x\in\mcX.
\]
Also, notice from Assumption~\ref{a1} that the domain of $F(x,\cdot)$ is compact for all $x\in\mcX$. By this, the above equality, and the strong duality, one has 
\beq \label{tfstar-ref}
f^*(x)=\min_{\lambda\in\cB^+_r}\max_{y}F(x,y)-\langle\lambda,d(x,y)\rangle \qquad \forall x\in\mcX.
\eeq
In addition, one can observe from \eqref{fstarx} that for all $x\in\mcX$, 
\[
f^*(x)=\max_y \min_{\lambda_\bfy} F(x,y)-\langle\lambda_\bfy,d(x,y)\rangle+\delta_{\bR^{\tm}_+}(\lambda_\bfy) \leq \min_{\lambda_\bfy}\max_y F(x,y)-\langle\lambda_\bfy,d(x,y)\rangle+\delta_{\bR^{\tm}_+}(\lambda_\bfy),
\]
where the inequality follows from the weak duality. This  together with \eqref{tfstar-ref} implies that \eqref{fstar-ref} holds.
\end{proof}

The next lemma provides an upper bound for $\{\lambda^k_\bfy\}_{k\in\bbK}$.

\begin{lemma}\label{l-lycnstr}
Suppose that Assumptions \ref{a1} and \ref{mfcq} hold. Let $\{\lambda^k_\bfy\}_{k\in\bbK}$ be generated by Algorithm \ref{AL-alg},  $D_\bfy$ and $\Delta$  be defined in   \eqref{mmax-D} and \eqref{def-r}, and $\tau$ and $\rho_k$ be given in  Algorithm \ref{AL-alg}. Then we have
\beq\label{ly-cnstr}
\rho_k^{-1}\|\lambda^k_\bfy\|^2\leq \|\lambda_\bfy^0\|^2+\frac{2(\Delta+D_\bfy)}{1-\tau}  \qquad \forall 0\leq k\in\bbK-1. 
\eeq
\end{lemma}

\begin{proof}
One can observe from \eqref{def-r} and Algorithm \ref{AL-alg} that$\Delta \geq 0$ and $\rho_0\geq1>\tau>0$, 
which imply that \eqref{ly-cnstr} holds for $k=0$. It remains to show that \eqref{ly-cnstr} holds for all $1\leq k\in\bbK-1$.

Since $(x^{t+1},y^{t+1})$ is an $\epsilon_t$-primal-dual stationary point of \eqref{AL-sub} for all $0\leq t\in\bbK-1$, it follows from Definition \ref{def2} that there exists some $u\in\partial_y\AL(x^{t+1},y^{t+1},\lambda^t_\bfx,\lambda^t_\bfy;\rho_t)$ with $\|u\|\leq\epsilon_t$. Notice from \eqref{AL} and \eqref{y-AL} that $\partial_y\AL(x^{t+1},y^{t+1},\lambda^t_\bfx,\lambda^t_\bfy;\rho_t)=\partial_y\AL_\bfy(x^{t+1},y^{t+1},\lambda^t_\bfy;\rho_t)$. 
Hence,  $u\in\partial_y\AL_\bfy(x^{t+1},y^{t+1},\lambda^t_\bfy;\rho_t)$. Also, observe from \eqref{prob}, \eqref{y-AL} and Assumption \ref{a1} that $\AL_\bfy(x^{t+1},\cdot,\lambda^t_\bfy;\rho_t)$ is concave. Using this, \eqref{mmax-D}, $u\in\partial_y\AL_\bfy(x^{t+1},y^{t+1},\lambda^t_\bfy;\rho_t)$ and $\|u\|\leq\epsilon_t$, we obtain
\begin{align*}
\AL_\bfy(x^{t+1},y,\lambda^t_\bfy;\rho_t)\leq&\ \AL_\bfy(x^{t+1},y^{t+1},\lambda^t_\bfy;\rho_t)+\langle u,y-y^{t+1}\rangle\\
\leq&\ \AL_\bfy(x^{t+1},y^{t+1},\lambda^t_\bfy;\rho_t)+D_\bfy\epsilon_t \qquad \forall y\in\mcY,
\end{align*}
which implies that
\beq\label{y-gap}
\max_y\AL_\bfy(x^{t+1},y,\lambda^t_\bfy;\rho_t)\leq\AL_\bfy(x^{t+1},y^{t+1},\lambda^t_\bfy;\rho_t)+D_\bfy\epsilon_t.
\eeq
By this, \eqref{y-AL} and \eqref{p-ineq}, one has
\begin{align*}
&\ f^*(x^{t+1})\overset{\eqref{p-ineq}}{\leq}\max_y\AL_\bfy(x^{t+1},y,\lambda^t_\bfy;\rho_t)\\
&\overset{\eqref{y-AL} \eqref{y-gap}}{\leq} F(x^{t+1},y^{t+1})-\frac{1}{2\rho_t}\left(\|[\lambda^t_\bfy+\rho_t d(x^{t+1},y^{t+1})]_+\|^2-\|\lambda^t_\bfy\|^2\right)+D_\bfy\epsilon_t\\
&\ \ \ = F(x^{t+1},y^{t+1})-\frac{1}{2\rho_t}\left(\|\lambda^{t+1}_\bfy\|^2-\|\lambda^t_\bfy\|^2\right)+D_\bfy\epsilon_t,
\end{align*}
where the equality follows from the relation $\lambda^{t+1}_\bfy=[\lambda^t_\bfy+\rho_t d(x^{t+1},y^{t+1})]_+$ (see Algorithm \ref{AL-alg}). Using the above inequality,\eqref{F-gap} and $\epsilon_t\leq 1$ (see Algorithm \ref{AL-alg}), we have
\begin{equation*}
\|\lambda^{t+1}_\bfy\|^2-\|\lambda^t_\bfy\|^2\leq2\rho_k\big(F(x^{t+1},y^{t+1})-f^*(x^{t+1})+D_\bfy\epsilon_t\big)\leq2\rho_t(\Delta+D_\bfy).
\end{equation*}
Summing up this inequality for $t=0,\dots,k-1$ with $1 \leq  k\in\bbK-1$ yields
\beq\label{l1-sum}
\|\lambda_\bfy^{k}\|^2\leq \|\lambda_\bfy^0\|^2+2(\Delta+D_\bfy)\sum_{t=0}^{k-1}\rho_t.
\eeq
Recall from Algorithm \ref{AL-alg} that $\rho_t=\epsilon_t^{-1}=\tau^{-t}$. Then we have $\sum_{t=0}^{k-1}\rho_t\leq\rho_{k-1}/(1-\tau)$. Using this, \eqref{l1-sum} and $\rho_{k}>\rho_{k-1}\geq1$ (see Algorithm \ref{AL-alg}), we obtain that for all $1 \leq k\in\bbK-1$,
\begin{align*}
\rho_k^{-1}\|\lambda_\bfy^k\|^2\leq\rho_k^{-1}\left( \|\lambda_\bfy^0\|^2+\frac{2(\Delta+D_\bfy)\rho_{k-1}}{1-\tau}\right)\leq\|\lambda_\bfy^0\|^2+\frac{2(\Delta+D_\bfy)}{1-\tau}.
\end{align*}
Hence, the conclusion holds as desired.
\end{proof}

The following lemma establishes an upper bound on $\|[d(x^{k+1},y^{k+1})]_+\|$ for $0\leq k\in\bbK-1$.

\begin{lemma}\label{l-ycnstr}
Suppose that Assumptions \ref{a1} and \ref{mfcq} hold. Let $D_\bfy$ and $\Delta$ be defined in \eqref{mmax-D} and \eqref{def-r}, and $\delta_d$ be given in Assumption \ref{mfcq}, and $\tau$ and $\rho_k$ be given in Algorithm \ref{AL-alg}. Suppose that $(x^{k+1},y^{k+1}, \lambda^{k+1}_\bfy)$ is generated by Algorithm \ref{AL-alg} for some $0\leq k\in\bbK-1$ with 
\beq\label{muk-bnd}
\rho_k\geq\frac{4\|\lambda_\bfy^0\|^2}{\delta_d^2}+\frac{8(\Delta+D_\bfy)}{\delta_d^2(1-\tau)}.
\eeq
Then we have
\beq\label{y-cnstr}
\|[d(x^{k+1},y^{k+1})]_+\|\leq\rho_k^{-1}\|\lambda^{k+1}_\bfy\|\leq2 \rho_k^{-1}\delta_d^{-1}(\Delta+D_\bfy).
\eeq
\end{lemma}

\begin{proof}
Suppose that $(x^{k+1},y^{k+1}, \lambda^{k+1}_\bfy)$ is generated by Algorithm \ref{AL-alg} for some $0\leq k\in\bbK-1$ with $\rho_k$ satisfying \eqref{muk-bnd}. 
Since $(x^{k+1}, y^{k+1})$ is an $\epsilon_k$-primal-dual stationary point of \eqref{AL-sub}, it follows from \eqref{AL} and Definition \ref{def2} that
\begin{align*}
\dist\left(0,\partial_yF(x^{k+1},y^{k+1})-\nabla_y d(x^{k+1},y^{k+1})[\lambda^k_\bfy+\rho_kd(x^{k+1},y^{k+1})]_+\right)\leq\epsilon_k.
\end{align*}
Besides, notice from Algorithm \ref{AL-alg} that  $\lambda^{k+1}_\bfy=[\lambda^k_\bfy+\rho_kd(x^{k+1},y^{k+1})]_+$. Hence, there exists some $u\in\partial_yF(x^{k+1},y^{k+1})$ such that
\beq\label{y-res}
\|u-\nabla_y d(x^{k+1},y^{k+1})\lambda^{k+1}_\bfy\|\leq\epsilon_k.
\eeq
By Assumption \ref{mfcq}(ii), there exists some $\hat y^{k+1}\in\mcY$ such that $-d_i(x^{k+1},\hat y^{k+1})\geq\delta_d$ for all $i$. 
Notice that $\langle\lambda^{k+1}_\bfy,\lambda^k_\bfy+\rho_kd(x^{k+1},y^{k+1})\rangle=\|[\lambda^k_\bfy+\rho_kd(x^{k+1},y^{k+1})]_+\|^2\geq0$, which implies that 
\beq \label{complim-ineq}
-\langle\lambda^{k+1}_\bfy,\rho_k^{-1}\lambda^k_\bfy\rangle \leq \langle\lambda^{k+1}_\bfy,d(x^{k+1},y^{k+1})\rangle.
\eeq
 Using these and \eqref{y-res}, we have
\begin{align}
& F(x^{k+1},\hat y^{k+1})-F(x^{k+1},y^{k+1})+\delta_d\|\lambda^{k+1}_\bfy\|_1-\rho_k^{-1}\langle\lambda^{k+1}_\bfy,\lambda^k_\bfy\rangle\nn\\
& \leq F(x^{k+1},\hat y^{k+1})-F(x^{k+1},y^{k+1})-\langle\lambda^{k+1}_\bfy,\rho_k^{-1}\lambda^k_\bfy+d(x^{k+1},\hat y^{k+1})\rangle\nn \\
& \overset{\eqref{complim-ineq}}\leq F(x^{k+1},\hat y^{k+1})-F(x^{k+1},y^{k+1})+\langle\lambda^{k+1}_\bfy,d(x^{k+1},y^{k+1})-d(x^{k+1},\hat y^{k+1}))\rangle\nn \\
& \leq \langle u, \hat y^{k+1}-y^{k+1}\rangle+\langle\nabla_yd(x^{k+1},y^{k+1})\lambda^{k+1}_\bfy,y^{k+1}-\hat y^{k+1}\rangle\nn \\
& = \langle u-\nabla_y d(x^{k+1},y^{k+1})\lambda^{k+1}_\bfy,y^{k+1}-\hat y^{k+1}\rangle\leq D_\bfy\epsilon_k, \label{bound-ineq}
\end{align}
where the first inequality is due to $\lambda^{k+1}_\bfy\geq 0$ and $-d_i(x^{k+1},\hat y^{k+1})\geq\delta_d$ for all $i$, the third inequality follows from  $u\in\partial_y F(x^{k+1}, y^{k+1})$, $\lambda^{k+1}_\bfy\geq 0$, the concavity of $F(x^{k+1},\cdot)$ and the convexity of $d_i(x^{k+1},\cdot)$, and the last inequality is due to \eqref{mmax-D} and \eqref{y-res}. 

In view of \eqref{Fhi} and \eqref{bound-ineq}, one has
\begin{align}
D_\bfy\epsilon_k+\Delta& \overset{\eqref{Fhi}}{\geq} D_\bfy\epsilon_k-F(x^{k+1},\hat y^{k+1})+F(x^{k+1},y^{k+1})\nn\\
&\overset{\eqref{bound-ineq}}{\geq} \delta_d\|\lambda^{k+1}_\bfy\|_1-\rho_k^{-1}\langle\lambda^{k+1}_\bfy,\lambda^k_\bfy\rangle\geq (\delta_d-\rho_k^{-1}\|\lambda^k_\bfy\|)\|\lambda^{k+1}_\bfy\|,\label{bnd-ineq}
\end{align}
where the last inequality is due to $\|\lambda^{k+1}_\bfy\|_1\geq\|\lambda^{k+1}_\bfy\|$. 
In addition, it follows from \eqref{ly-cnstr} and \eqref{muk-bnd} that
\[
\delta_d-\rho_k^{-1}\|\lambda^k_\bfy\|\overset{\eqref{ly-cnstr}}{\geq}\delta_d-\sqrt{\rho_k^{-1}\left(\|\lambda_\bfy^0\|^2+\frac{2(F_{\rm hi}-f^*_{\rm low}+D_\bfy)}{1-\tau}\right)}\ \overset{\eqref{muk-bnd}}{\geq}\frac{1}{2}\delta_d,
\]
which together with \eqref{bnd-ineq} yields
\[
\frac{1}{2}\delta_d\|\lambda^{k+1}_\bfy\|\leq(\delta_d-\rho_k^{-1}\|\lambda^k_\bfy\|)\|\lambda^{k+1}_\bfy\|\overset{\eqref{bnd-ineq}}{\leq} D_\bfy\epsilon_k+\Delta.
\]
The conclusion then follows from this, $\epsilon_k\leq 1$, and  the relations 
\[
\|[d(x^{k+1},y^{k+1})]_+\|\leq\rho_k^{-1}\|[\lambda^k_\bfy+\rho_k d(x^{k+1},y^{k+1})]_+\|=\rho_k^{-1}\|\lambda^{k+1}_\bfy\|.
\]
\end{proof}

The next lemma provides an upper bound on the amount of violation of the conditions in  \eqref{kkt1}, \eqref{kkt2} and \eqref{kkt4} at $(x,y,\lambda_\bfx,\lambda_\bfy)=(x^{k+1},y^{k+1},\tl^{k+1}_\bfx,\lambda^{k+1}_\bfy)$ for $0\leq k\in\bbK-1$, where $\tl^{k+1}_\bfx$ is given below.

\begin{lemma}\label{l-subdcnstr}
Suppose that Assumptions \ref{a1} and \ref{mfcq} hold. Let  $D_\bfy$ and $\Delta$ be defined in  \eqref{mmax-D} and \eqref{def-r}, and $\delta_d$ be given in Assumption \ref{mfcq}, and $\tau$, $\epsilon_k$, $\rho_k$ and $\lambda_\bfy^0$ be given in Algorithm \ref{AL-alg}.
Suppose that $(x^{k+1},y^{k+1}, \lambda^{k+1}_\bfx,\lambda^{k+1}_\bfy)$ is generated by Algorithm \ref{AL-alg} for  some $0\leq k\in\bbK-1$ with 
\beq\label{muk-1}
\rho_k\geq\frac{4\|\lambda_\bfy^0\|^2}{\delta_d^2\tau}+\frac{8(\Delta+D_\bfy)}{\delta_d^2\tau(1-\tau)}.
\eeq
Let 
\beq \label{def-tlx} 
 \tl^{k+1}_\bfx=[\lambda^k_\bfx+\rho_kc(x^{k+1})]_+.
\eeq
Then we have
\begin{align}
& \dist(0,\partial_x F(x^{k+1},y^{k+1})+\nabla c(x^{k+1})\tl^{k+1}_\bfx -\nabla_xd(x^{k+1},y^{k+1})\lambda^{k+1}_\bfy) \leq \epsilon_k, \label{pF-x}\\
& \dist\left(0,\partial_yF(x^{k+1},y^{k+1})-\nabla_y d(x^{k+1},y^{k+1})\lambda^{k+1}_\bfy\right)\leq\epsilon_k, \label{pF-y}\\
& \|[d(x^{k+1},y^{k+1})]_+\|\leq 2\rho_k^{-1}\delta_d^{-1}(\Delta+D_\bfy), \label{d-feas}\\ 
&|\langle\lambda^{k+1}_\bfy,d(x^{k+1},y^{k+1})\rangle|\leq 2\rho_k^{-1}\delta_d^{-1}(\Delta+D_\bfy) \max\{\|\lambda_\bfy^0\|,\ 2\delta_d^{-1}(\Delta+D_\bfy)\}. \label{d-complim}
\end{align}
\end{lemma}

\begin{proof}
Suppose that $(x^{k+1},y^{k+1}, \lambda^{k+1}_\bfx,\lambda^{k+1}_\bfy)$ is generated by Algorithm \ref{AL-alg} for  some $0\leq k\in\bbK-1$ with $\rho_k$ satisfying \eqref{muk-1}. Since $(x^{k+1},y^{k+1})$ is an $\epsilon_k$-primal-dual stationary point of \eqref{AL-sub}, it then follows from Definition \ref{def2} that
\beq\label{stationary}
\dist\big(0,\partial_x\AL(x^{k+1},y^{k+1},\lambda^k_\bfx,\lambda^k_\bfy;\rho_k)\big)\leq\epsilon_k,\  \dist\big(0,\partial_y\AL(x^{k+1},y^{k+1},\lambda^k_\bfx,\lambda^k_\bfy;\rho_k)\big)\leq\epsilon_k.
\eeq
Observe from Algorithm \ref{AL-alg} that $\lambda^{k+1}_\bfy=[\lambda^k_\bfy+\rho_kd(x^{k+1},y^{k+1})]_+$. In view of this, \eqref{AL} and \eqref{def-tlx}, one has
\begin{align*}
\partial_x\AL(x^{k+1},y^{k+1},\lambda^k_\bfx,\lambda^k_\bfy;\rho_k)=&\ \partial_x F(x^{k+1},y^{k+1})+\nabla c(x^{k+1})[\lambda^k_\bfx+\rho_kc(x^{k+1})]_+\\
&\ -\nabla_xd(x^{k+1},y^{k+1})[\lambda^k_\bfy+\rho_kd(x^{k+1},y^{k+1})]_+\\
=&\ \partial_x F(x^{k+1},y^{k+1})+\nabla c(x^{k+1})\tl^{k+1}_\bfx -\nabla_xd(x^{k+1},y^{k+1})\lambda^{k+1}_\bfy, \\
\partial_y\AL(x^{k+1},y^{k+1},\lambda^k_\bfx,\lambda^k_\bfy;\rho_k)=& \ \partial_y F(x^{k+1},y^{k+1})-\nabla_y d(x^{k+1},y^{k+1})\lambda^{k+1}_\bfy.
\end{align*}
These relations together with  \eqref{stationary} imply that \eqref{pF-x} and \eqref{pF-y} hold.

Notice from Algorithm \ref{AL-alg} that $0<\tau<1$, which together with \eqref{muk-1} implies that \eqref{muk-bnd} holds for $\rho_k$. It then follows that \eqref{y-cnstr} holds, which immediately yields \eqref{d-feas} and 
\beq \label{lambday-bnd1}
\|\lambda^{k+1}_\bfy\|\leq2\delta_d^{-1}(\Delta+D_\bfy).
\eeq
Claim that 
\beq \label{lambday-bnd2}
\|\lambda^k_\bfy\|\leq\max\{\|\lambda_\bfy^0\|,\ 2\delta_d^{-1}(\Delta+D_\bfy)\}.
\eeq
Indeed, \eqref{lambday-bnd2} clearly holds if $k=0$.  We now assume that $k>0$. Notice from 
Algorithm \ref{AL-alg} that $\rho_{k-1}=\tau\rho_k$, which together with \eqref{muk-1} implies that \eqref{muk-bnd} holds with $k$ replaced by $k-1$. By this and Lemma \ref{l-ycnstr} with $k$ replaced by $k-1$, one can conclude that $\|\lambda^k_\bfy\|\leq2\delta_d^{-1}(\Delta+D_\bfy)$ and hence \eqref{lambday-bnd2} holds.

We next show that \eqref{d-complim} holds. Indeed, by  $\lambda^{k+1}_\bfy\geq 0$, \eqref{complim-ineq}, \eqref{d-feas}, \eqref{lambday-bnd1} and \eqref{lambday-bnd2}, one has 
\begin{align*}
\langle\lambda^{k+1}_\bfy,d(x^{k+1},y^{k+1})\rangle & \ \ \leq \ \langle\lambda^{k+1}_\bfy,[d(x^{k+1},y^{k+1})]_+\rangle \leq\|\lambda^{k+1}_\bfy\|\|[d(x^{k+1},y^{k+1})]_+\|\\ 
& \overset{\eqref{d-feas}\eqref{lambday-bnd1}}\leq4\rho_k^{-1}\delta_d^{-2}(\Delta+D_\bfy)^2,\\
\langle\lambda^{k+1}_\bfy,d(x^{k+1},y^{k+1})\rangle& \overset{\eqref{complim-ineq}}\geq\langle\lambda^{k+1}_\bfy,-\rho_k^{-1}\lambda^k_\bfy\rangle\geq-\rho_k^{-1}\|\lambda^{k+1}_\bfy\|\|\lambda^k_\bfy\|\\
&\ \geq-2\rho_k^{-1}\delta_d^{-1}(\Delta+D_\bfy) \max\{\|\lambda_\bfy^0\|,\ 2\delta_d^{-1}(\Delta+D_\bfy)\}.
\end{align*}
These relations imply that \eqref{d-complim} holds.
\end{proof}

The following lemma provides an upper bound on $\max_y\AL(x^k_{\rm init},y,\lambda^k_\bfx,\lambda^k_\bfy;\rho_k)$ for $0\leq k\in\bbK-1$, which will subsequently be used to derive an upper bound for  $\max_y\AL(x^{k+1},y,\lambda^k_\bfx,\lambda^k_\bfy;\rho_k)$.

\begin{lemma}
Suppose that Assumptions \ref{a1}, \ref{knownfeas} and \ref{mfcq} hold. Let $\{(\lambda^k_\bfx,\lambda^k_\bfy)\}_{k\in\bbK}$ be generated by Algorithm \ref{AL-alg}, $\AL$, $D_\bfy$,  $F_{\rm hi}$ and $\Delta$ be defined in \eqref{AL}, \eqref{mmax-D}, \eqref{Fhi} and \eqref{def-r}, and $\tau$, $\rho_k$, $\Lambda$ and $x^k_{\rm init}$ be given in Algorithm \ref{AL-alg}. Then for all $0\leq k\in\bbK-1$, we have
\beq\label{init-upper}
\max_y\AL(x^k_{\rm init},y,\lambda^k_\bfx,\lambda^k_\bfy;\rho_k)\leq \Delta+F_{\rm hi}+\Lambda+\frac{1}{2}(1+\|\lambda_\bfy^0\|^2)+\frac{\Delta+D_\bfy}{1-\tau}.
\eeq
\end{lemma}

\begin{proof}
In view of \eqref{x-AL}, \eqref{xinit},  \eqref{Fhi} and $\|\lambda^k_\bfx\|\leq\Lambda$ (see Algorithm \ref{AL-alg}), one has 
\begin{align}
\AL_\bfx(x^k_{\rm init},y^k,\lambda^k_\bfx;\rho_k) & \overset{\eqref{xinit}}\leq\AL_\bfx(x_\bff,y^k,\lambda^k_\bfx;\rho_k) \overset{\eqref{x-AL}}= F(x_\bff,y^k)+\frac{1}{2\rho_k}\big(\|[\lambda^k_\bfx+\rho_k c(x_\bff)]_+\|^2-\|\lambda^k_\bfx\|^2\big)\nn\\
&\leq F(x_\bff,y^k)+\frac{1}{2\rho_k}\big((\|\lambda^k_\bfx\|+\rho_k\| [c(x_\bff)]_+\|)^2-\|\lambda^k_\bfx\|^2\big)\nn\\
&= F(x_\bff,y^k)+\|\lambda^k_\bfx\|\|[c(x_\bff)]_+\|+\frac{1}{2}\rho_k\| [c(x_\bff)]_+\|^2 \nn \\ 
&\overset{\eqref{Fhi}} \leq F_{\rm hi}+\Lambda\| [c(x_\bff)]_+\|+\frac{1}{2}\rho_k\| [c(x_\bff)]_+\|^2.\label{init-ineq}
\end{align}
In addition, one can observe from Algorithm \ref{AL-alg} that $\epsilon_k>\tau\varepsilon$ for all $0\leq k\in\bbK-1$. By this and the choice of $\rho_k$ in Algorithm \ref{AL-alg}, we obtain that $\rho_k=\epsilon_k^{-1}<\tau^{-1}\varepsilon^{-1}$ for all $0\leq k\in\bbK-1$. It then follows from this, \eqref{AL},  \eqref{x-AL}, \eqref{def-r}, \eqref{ly-cnstr}, \eqref{init-ineq}, $\|[c(x_\bff)]_+\|\leq \sqrt{\varepsilon}\leq1$, and the Lipschitz continuity of $F$ that
\begin{align*}
&\ \max_y\AL(x^k_{\rm init},y,\lambda^k_\bfx,\lambda^k_\bfy;\rho_k)\overset{\eqref{AL} \eqref{x-AL}}=\max_y\left\{\AL_\bfx(x^k_{\rm init},y,\lambda^k_\bfx;\rho_k)-\frac{1}{2\rho_k}\left(\|[\lambda^k_\bfy+\rho_k d(x^k_{\rm init},y)]_+\|^2-\|\lambda^k_\bfy\|^2\right)\right\}\\
&\leq \max_y\left\{\AL_\bfx(x^k_{\rm init},y,\lambda^k_\bfx;\rho_k)+\frac{1}{2\rho_k}\|\lambda^k_\bfy\|^2\right\}\\
&\overset{\eqref{x-AL}}=\max_y\left\{F(x^k_{\rm init},y)-F(x^k_{\rm init},y^k)+\AL_\bfx(x^k_{\rm init},y^k,\lambda^k_\bfx;\rho_k)+\frac{1}{2\rho_k}\|\lambda^k_\bfy\|^2\right\}\\
&\overset{\eqref{def-r}}\leq \Delta +\AL_\bfx(x^k_{\rm init},y^k,\lambda^k_\bfx;\rho_k)+\frac{1}{2\rho_k}\|\lambda^k_\bfy\|^2\\
&\leq \Delta+F_{\rm hi}+\Lambda\| [c(x_\bff)]_+\|+\frac{1}{2}\rho_k\| [c(x_\bff)]_+\|^2+\frac{1}{2}\|\lambda_\bfy^0\|^2+\frac{\Delta+D_\bfy}{1-\tau}\\
&\leq \Delta+F_{\rm hi}+\Lambda+\frac{1}{2}(\tau^{-1}+\|\lambda_\bfy^0\|^2)+\frac{\Delta+D_\bfy}{1-\tau},
\end{align*}
where  the third inequality follows from  \eqref{ly-cnstr} and \eqref{init-ineq}, and the last inequality follows from $\rho_k<\tau^{-1}\varepsilon^{-1}$ and $\|[c(x_\bff)]_+\|\leq \sqrt{\varepsilon}\leq1$.
\end{proof}

The next lemma shows that an approximate primal-dual stationary point  of \eqref{AL-sub} is found at each iteration of Algorithm~\ref{AL-alg}, and also provides an estimate of operation complexity for finding it.

\begin{lemma}\label{l-subp}
Suppose that Assumptions \ref{a1}, \ref{knownfeas} and \ref{mfcq} hold. Let $D_\bfx$, $D_\bfy$, $L_k$,  $F_{\rm hi}$ and $\Delta$ be defined in  \eqref{mmax-D}, \eqref{Lk}, \eqref{Fhi} and \eqref{def-r},  $\tau$, $\epsilon_k$,  $\rho_k$, $\Lambda$ and $\lambda_\bfy^0$ be given in Algorithm \ref{AL-alg}, and
\begin{align}
\alpha_k=&\ \min\left\{1,\sqrt{4\epsilon_k/(D_\bfy L_k)}\right\},\label{mmax-omega}\\
\delta_k=&\ (2+\alpha_k^{-1})L_k D_\bfx^2+\max\left\{\epsilon_k/D_\bfy,\alpha_k L_k/4\right\}D_\bfy^2,\label{mmax-zeta}\\
M_k=&\ \frac{16\max\left\{1/(2L_k),\min\left\{D_\bfy/\epsilon_k,4/(\alpha_k L_k)\right\}\right\}\rho_k}{\left[(3L_k+\epsilon_k/(2D_\bfy))^2/\min\{L_k,\epsilon_k/(2D_\bfy)\}+ 3L_k+\epsilon_k/(2D_\bfy)\right]^{-2}\epsilon_k^2}\nn\\
&\times\Bigg(\delta_k+2\alpha_k^{-1}\bigg(\Delta+\frac{\Lambda^2}{2\rho_k}+\frac{3}{2}\|\lambda_\bfy^0\|^2+\frac{3(\Delta+D_\bfy)}{1-\tau}+\rho_kd_{\rm hi}^2+\frac{\epsilon_k D_\bfy}{4}+L_k D_\bfx^2\bigg)\Bigg)\label{mmax-Mk}\\
T_k=&\ \Bigg\lceil16\left(2\Delta+\Lambda+\frac{1}{2}(\tau^{-1}+\|\lambda_\bfy^0\|^2)+\frac{\Delta+D_\bfy}{1-\tau}+\frac{\Lambda^2}{2\rho_k}+\frac{\epsilon_k D_\bfy}{4}\right) L_k\epsilon_k^{-2}\nn\\
&\ +8(1+4D_\bfy^2L_k^2\epsilon_k^{-2})\rho_k^{-1}-1\Bigg\rceil_+,\label{mmax-Tk}\\
N_k=&\ \left(\left\lceil96\sqrt{2}\left(1+\left(24L_k+4\epsilon_k/D_\bfy\right)L_k^{-1}\right)\right\rceil+2\right)\max\left\{2,\sqrt{D_\bfy L_k\epsilon_k^{-1}}\right\}\notag\\
&\ \times\left((T_k+1)(\log M_k)_++T_k+1+2T_k\log(T_k+1) \right).\label{mmax-N}
\end{align}
Then for all $0\leq k\in\bbK-1$, Algorithm~\ref{AL-alg} finds an $\epsilon_k$-primal-dual stationary point $(x^{k+1},y^{k+1})$ of problem \eqref{AL-sub} satisfying 
\begin{align}
\max_y\AL(x^{k+1},y,\lambda^k_\bfx,\lambda^k_\bfy;\rho_k)\leq&\ \Delta+F_{\rm hi}+\Lambda+\frac{1}{2}(\tau^{-1}+\|\lambda_\bfy^0\|^2)+\frac{\Delta+D_\bfy}{1-\tau}\nn\\
&\ +\frac{\epsilon_k D_\bfy}{4}+\frac{1}{2\rho_k}\left(L_k^{-1}\epsilon_k^2+4D_\bfy^2L_k\right).\label{upperbnd}
\end{align}
Moreover, the total number of evaluations of $\nabla f$, $\nabla c$, $\nabla d$ and proximal operator of $p$ and $q$ performed in iteration $k$ of Algorithm~\ref{AL-alg} is no more than $N_k$, respectively.
\end{lemma}

\begin{proof}
Observe from \eqref{prob} and \eqref{AL} that problem \eqref{AL-sub} can be viewed as
\[
\min_x\max_y\{h(x,y)+p(x)-q(y)\},
\]
where
\[
h(x,y)=f(x,y)+\frac{1}{2\rho _k}\left(\|[\lambda^k_\bfx+\rho_k c(x)]_+\|^2-\|\lambda^k_\bfx\|^2\right)-\frac{1}{2\rho_k}\left(\|[\lambda^k_\bfy+\rho_k d(x,y)]_+\|^2-\|\lambda^k_\bfy\|^2\right).
\]
Notice that
\begin{align*}
& \nabla_x h(x,y)=\nabla_x f(x,y)+\nabla c(x)[\lambda^k_\bfx+\rho_kc(x)]_++\nabla_x d(x,y)[\lambda^k_\bfy+\rho_kd(x,y)]_+, \\
& \nabla_y h(x,y)=\nabla_y f(x,y)+\nabla_y d(x,y)[\lambda^k_\bfy+\rho_kd(x,y)]_+.
\end{align*}
It follows from Assumption \ref{a1}(iii) that 
\[
\|\nabla c(x)\|\leq L_c,\quad\|\nabla d(x,y)\| \leq L_d \qquad \forall (x,y)\in\mcX\times\mcY.
\]
In view of the above relations, \eqref{cdhi} and Assumption \ref{a1}, one can observe that  $\nabla c(x)[\lambda^k_\bfx+\rho_kc(x)]_+$ is $(\rho_kL_c^2+\rho_kc_{\rm hi}L_{\nabla c}+\|\lambda^k_\bfx\| L_{\nabla c})$-Lipschitz continuous on $\mcX$, and $\nabla d(x,y)[\lambda^k_\bfy+\rho_kd(x,y)]_+$ is $(\rho_kL_d^2+\rho_kd_{\rm hi}L_{\nabla d}+\|\lambda^k_\bfy\|L_{\nabla d})$-Lipschitz continuous on $\mcX\times\mcY$. Using these and the fact that $\nabla f(x,y)$ is $L_{\nabla f}$-Lipschitz continuous on $\mcX\times\mcY$, we can see that $h(x,y)$ is $L_k$-smooth on $\mcX\times\mcY$ for all $0\leq k\in\bbK-1$, where $L_k$ is given in \eqref{Lk}. Consequently, it follows from Theorem \ref{mmax-thm} that Algorithm \ref{mmax-alg2} can be suitably applied to problem \eqref{AL-sub} for finding an $\epsilon_k$-primal-dual stationary point $(x^{k+1},y^{k+1})$ of it.

In addition, by \eqref{AL},\eqref{F-gap}, \eqref{y-AL}, \eqref{p-ineq} and $\|\lambda^k_\bfx\|\leq\Lambda$ (see Algorithm \ref{AL-alg}), one has
\begin{align}
&\ \min_x\max_y\AL(x,y,\lambda^k_\bfx,\lambda^k_\bfy;\rho_k)\overset{\eqref{AL} \eqref{y-AL}}=\min_x\max_y\left\{\AL_\bfy(x,y,\lambda^k_\bfy;\rho_k)+\frac{1}{2\rho_k}\left(\|[\lambda^k_\bfx+\rho_k c(x)]_+\|^2-\|\lambda^k_\bfx\|^2\right)\right\}\nn\\
&\overset{\eqref{p-ineq}}{\geq}\min_x \left\{f^*(x)+\frac{1}{2\rho_k}\left(\|[\lambda^k_\bfx+\rho_k c(x)]_+\|^2-\|\lambda^k_\bfx\|^2\right)\right\} \overset{\eqref{F-gap}}{\geq} F_{\rm low}-\frac{1}{2\rho_k}\|\lambda^k_\bfx\|^2\geq F_{\rm low}-\frac{\Lambda^2}{2\rho_k}.\label{e1}
\end{align}
Let $(x^*,y^*)$ be an optimal solution of \eqref{prob}. It then follows that $c(x^*)\leq0$. Using this, \eqref{AL}, \eqref{Fhi} and \eqref{ly-cnstr}, we obtain that
\begin{align}
&\ \ \min_x\max_y\AL(x,y,\lambda^k_\bfx,\lambda^k_\bfy;\rho_k)\leq\max_y\AL(x^*,y,\lambda^k_\bfx,\lambda^k_\bfy;\rho_k)\nn\\
&\ \overset{\eqref{AL}}{=}\max_y\left\{F(x^*,y)+\frac{1}{2\rho_k}\left(\|[\lambda^k_\bfx+\rho_kc(x^*)]_+\|^2-\|\lambda^k_\bfx\|^2\right)-\frac{1}{2\rho_k}\left(\|[\lambda^k_\bfy+\rho_k d(x^*,y)]_+\|^2-\|\lambda^k_\bfy\|^2\right)\right\}\nn\\
&\ \leq\ \max_y\left\{F(x^*,y)-\frac{1}{2\rho_k}\left(\|[\lambda^k_\bfy+\rho_k d(x^*,y)]_+\|^2-\|\lambda^k_\bfy\|^2\right)\right\}\nn\\
&\  \overset{\eqref{Fhi}}{\leq} F_{\rm hi}+\frac{1}{2\rho_k}\|\lambda^k_\bfy\|^2\overset{ \eqref{ly-cnstr}}{\leq}F_{\rm hi}+\frac{1}{2}\|\lambda_\bfy^0\|^2+\frac{\Delta+D_\bfy}{1-\tau},\label{e2}
\end{align}
where the second inequality is due to $c(x^*)\leq0$. Moreover, it follows from this, \eqref{AL}, \eqref{cdhi}, \eqref{Fhi}, \eqref{ly-cnstr}, $\lambda^k_\bfy\in\bR_+^{\tm}$ and $\|\lambda^k_\bfx\|\leq\Lambda$ that
\begin{align}
& \min_{(x,y)\in\mcX\times\mcY}\AL(x,y,\lambda^k_\bfx,\lambda^k_\bfy;\rho_k)\overset{\eqref{AL}}{\geq}\min_{(x,y)\in\mcX\times\mcY}\left\{F(x,y)-\frac{1}{2\rho_k}\|\lambda^k_\bfx\|^2-\frac{1}{2\rho_k}\|[\lambda^k_\bfy+\rho_kd(x,y)]_+\|^2\right\}\nn\\
&\ \geq\min_{(x,y)\in\mcX\times\mcY}\left\{F(x,y)-\frac{1}{2\rho_k}\|\lambda^k_\bfx\|^2-\frac{1}{2\rho_k}\left(\|\lambda^k_\bfy\|+\rho_k\|[d(x,y)]_+\|\right)^2\right\}\nn\\
&\ \geq\min_{(x,y)\in\mcX\times\mcY}\left\{F(x,y)-\frac{1}{2\rho_k}\|\lambda^k_\bfx\|^2-\rho_k^{-1}\|\lambda^k_\bfy\|^2-\rho_k\|[d(x,y)]_+\|^2\right\}\nn\\
&\ \geq F_{\rm low}-\frac{\Lambda^2}{2\rho_k}-\|\lambda_\bfy^0\|^2-\frac{2(\Delta+D_\bfy)}{1-\tau}-\rho_kd_{\rm hi}^2,\label{e3}
\end{align}
where the second inequality is due to $\lambda^k_\bfy\in\bR_+^{\tm}$ and the last inequality is due to \eqref{cdhi}, \eqref{Fhi}, \eqref{ly-cnstr} and $\|\lambda^k_\bfx\|\leq\Lambda$.

To complete the rest of the proof, let
\begin{align}
&H(x,y)=\AL(x,y,\lambda^k_\bfx,\lambda^k_\bfy;\rho_k),\quad H^*=\min_x\max_y\AL(x,y,\lambda^k_\bfx,\lambda^k_\bfy;\rho_k),\label{Hk1}\\
&H_{\rm low}=\min_{(x,y)\in\mcX\times\mcY}\AL(x,y,\lambda^k_\bfx,\lambda^k_\bfy;\rho_k).\label{Hk2}
\end{align}
In view of these, \eqref{init-upper}, \eqref{e1}, \eqref{e2}, \eqref{e3}, we obtain that
\begin{align*}
&\max_yH(x^k_{\rm init},y)\overset{\eqref{init-upper}}{\leq}\Delta+F_{\rm hi}+\Lambda+\frac{1}{2}(\tau^{-1}+\|\lambda_\bfy^0\|^2)+\frac{\Delta+D_\bfy}{1-\tau},\\
&F_{\rm low}-\frac{\Lambda^2}{2\rho_k}\overset{\eqref{e1}}{\leq} H^*\overset{\eqref{e2}}{\leq} F_{\rm hi}+\frac{1}{2}\|\lambda_\bfy^0\|^2+\frac{\Delta+D_\bfy}{1-\tau},\\
&H_{\rm low}\overset{\eqref{e3}}{\geq}F_{\rm low}-\frac{\Lambda^2}{2\rho_k}-\|\lambda_\bfy^0\|^2-\frac{2(\Delta+D_\bfy)}{1-\tau}-\rho_kd_{\rm hi}^2.
\end{align*}
Using these, \eqref{def-r}, and Theorem \ref{mmax-thm} with $\hat x^0=x^k_{\rm init}$, $\epsilon=\epsilon_k$, $\hat\epsilon_0=\epsilon_k/(2\sqrt{\rho_k})$, $L_{\nabla\bh}=L_k$, 
and $\bH$, $\bH^*$, $\bH_{\rm low}$ given in \eqref{Hk1} and \eqref{Hk2}, we can conclude that Algorithm \ref{mmax-alg2} performs at most $N_k$ evaluations of $\nabla f$, $\nabla c$, $\nabla d$ and proximal operator of $p$ and $q$ for finding an $\epsilon_k$-primal-dual stationary point of problem \eqref{AL-sub} satisfying \eqref{upperbnd}.
\end{proof}

The following lemma provides an upper bound on the violation of the conditions in \eqref{kkt3} at $(x,\lambda_\bfx)=(x^{k+1},\tl^{k+1}_\bfx)$ for $0\leq k\in\bbK-1$, where $\tl^{k+1}_\bfx$ is given below.

\begin{lemma}\label{l-xcnstr2}
Suppose that Assumptions \ref{a1}, \ref{knownfeas} and \ref{mfcq} hold. Let $D_\bfy$, $\Delta$ and $L$ be defined in \eqref{mmax-D},  \eqref{def-r} and \eqref{hL}, $L_F$, $L_c$, $\delta_c$ and $\theta$ be given in Assumption \ref{mfcq}, and $\tau$, $\rho_k$, $\Lambda$ and $\lambda_\bfy^0$ be given in Algorithm \ref{AL-alg}. Suppose that $(x^{k+1},\lambda^{k+1}_\bfx)$ is generated by Algorithm \ref{AL-alg} for some $0\leq k\in\bbK-1$ with 
\begin{align}
\rho_k \geq \max\Bigg\{&\theta^{-1}\Lambda, \theta^{-2}\Big\{4\Delta+2\Lambda+\tau^{-1}+\|\lambda_\bfy^0\|^2+\frac{2(\Delta+D_\bfy)}{1-\tau} +\frac{D_\bfy}{2} +L_c^{-2} +4D_\bfy^2L+\Lambda^2\Big\}, \nn \\ 
& \frac{4\|\lambda_\bfy^0\|^2}{\delta_d^2\tau}+\frac{8(\Delta+D_\bfy)}{\delta_d^2\tau(1-\tau)}\Bigg\}. \label{rhok-bnd}
\end{align}
Let 
\beq \label{def-tlx1} 
 \tl^{k+1}_\bfx=[\lambda^k_\bfx+\rho_kc(x^{k+1})]_+.
\eeq
Then we have
\begin{align}
&\|[c(x^{k+1})]_+\|\leq \rho_k^{-1}\delta_c^{-1}\left(L_F +2L_d\delta_d^{-1}(\Delta+D_\bfy)+1\right), \label{x-cnstr2} \\
& |\langle\tl^{k+1}_\bfx,c(x^{k+1})\rangle| \leq \rho^{-1}_k \delta_c^{-1}(L_F +2L_d\delta_d^{-1}(\Delta+D_\bfy)+1)\max\{\delta_c^{-1}(L_F +2L_d\delta_d^{-1}(\Delta+D_\bfy)+1), \Lambda\}. \label{c-complim}
\end{align}
\end{lemma}

\begin{proof}
One can observe from \eqref{AL},\eqref{F-gap}, \eqref{y-AL} and \eqref{p-ineq} that
\begin{align*}
\max_y\AL(x^{k+1},y,\lambda^k_\bfx,\lambda^k_\bfy;\rho_k)=\ &\ \max_y\AL_\bfy(x^{k+1},y,\lambda^k_\bfy;\rho_k)+\frac{1}{2\rho_k}\left(\|[\lambda^k_\bfx+\rho_k c(x^{k+1})]_+\|^2-\|\lambda^k_\bfx\|^2\right)\\
\overset{\eqref{p-ineq}}{\geq}&\ f^*(x^{k+1})+\frac{1}{2\rho_k}\left(\|[\lambda^k_\bfx+\rho_k c(x^{k+1})]_+\|^2-\|\lambda^k_\bfx\|^2\right)\\
\overset{\eqref{F-gap}}{\geq}\ &\ F_{\rm low}+\frac{1}{2\rho_k}\left(\|[\lambda^k_\bfx+\rho_k c(x^{k+1})]_+\|^2-\|\lambda^k_\bfx\|^2\right).
\end{align*}
By this inequality, \eqref{upperbnd} and $\|\lambda^k_\bfx\|\leq\Lambda$, one has
\begin{align*}
&\|[\lambda^k_\bfx+\rho_kc(x^{k+1})]_+\|^2\leq 2\rho_k\max_y\AL(x^{k+1},y,\lambda^k_\bfx,\lambda^k_\bfy;\rho_k)-2\rho_k F_{\rm low}+\|\lambda^k_\bfx\|^2\\
&\leq2\rho_k\max_y\AL(x^{k+1},y,\lambda^k_\bfx,\lambda^k_\bfy;\rho_k)-2\rho_k F_{\rm low}+\Lambda^2\\
&\overset{\eqref{upperbnd}}\leq 2\rho_k\Delta+2\rho_kF_{\rm hi}+2\rho_k\Lambda+\rho_k(\tau^{-1}+\|\lambda_\bfy^0\|^2)+\frac{2\rho_k(\Delta+D_\bfy)}{1-\tau}+\frac{\rho_k\epsilon_k D_\bfy}{2}\\
&\qquad +L_k^{-1}\epsilon_k^2+4D_\bfy^2L_k-2\rho_k F_{\rm low}+\Lambda^2.
\end{align*}
This together with \eqref{def-r} and $\rho_k^2\|[c(x^{k+1})]_+\|^2\leq\|[\lambda^k_\bfx+\rho_kc(x^{k+1})]_+\|^2$ implies that
\begin{align} 
\|[c(x^{k+1})]_+\|^2\leq&\ \rho_k^{-1}\left(4\Delta+2\Lambda+\tau^{-1}+\|\lambda_\bfy^0\|^2+\frac{2(\Delta+D_\bfy)}{1-\tau}+\frac{\epsilon_k D_\bfy}{2}\right)\nn\\
&\ +\rho_k^{-2}\left(L_k^{-1}\epsilon_k^2+4D_\bfy^2L_k+\Lambda^2\right). \label{x-cnstr}
\end{align}
In addition, we observe from \eqref{Lk}, \eqref{hL}, \eqref{ly-cnstr}, $\rho_k\geq1$ and $\|\lambda^k_\bfx\|\leq\Lambda$ that for all $0\leq k\leq K$,
\begin{align}
\rho_kL_c^2 & \leq L_k= L_{\nabla f}+\rho_kL_c^2+\rho_kc_{\rm hi}L_{\nabla c}+\|\lambda^k_\bfx\|L_{\nabla c}+\rho_kL_d^2+\rho_kd_{\rm hi}L_{\nabla d}+\|\lambda^k_\bfy\|L_{\nabla d}\nn\\
&\leq L_{\nabla f}+\rho_kL_c^2+\rho_kc_{\rm hi}L_{\nabla c}+\Lambda L_{\nabla c}+\rho_kL_d^2+\rho_kd_{\rm hi}L_{\nabla d}\nn\\
&\ \ \ \ +L_{\nabla d}\sqrt{\rho_k\left(\|\lambda_\bfy^0\|^2+\frac{2(\Delta+D_y)}{1-\tau}\right)}\leq\rho_kL.\label{L-ineq}
\end{align}
 Using this relation, \eqref{rhok-bnd}, \eqref{x-cnstr},  $\rho_k\geq1$ and $\epsilon_k \leq 1$, we have
\begin{align*}
\|[c(x^{k+1})]_+\|^2\leq&\ \rho_k^{-1}\left(4\Delta+2\Lambda+\tau^{-1}+\|\lambda_\bfy^0\|^2+\frac{2(\Delta+D_\bfy)}{1-\tau}+\frac{\epsilon_k D_\bfy}{2}\right)\nn\\
& \ +\rho_k^{-2}\left((\rho_kL_c^2)^{-1}\epsilon_k^2+4\rho_kD_\bfy^2L+\Lambda^2\right)\\
\leq\ &\ \rho_k^{-1}\left(4\Delta+2\Lambda+\tau^{-1}+\|\lambda_\bfy^0\|^2+\frac{2(\Delta+D_\bfy)}{1-\tau}+\frac{D_\bfy}{2}\right)\nn\\
& \ +\rho_k^{-1}\left(L_c^{-2}+4D_\bfy^2L+\Lambda^2\right)\overset{\eqref{rhok-bnd}}{\leq}\theta^2,
\end{align*}
which together with \eqref{def-cAcS} implies that $x^{k+1}\in\cF(\theta)$.  

It follows from $x^{k+1}\in\cF(\theta)$ and Assumption \ref{mfcq}(i) that there exists some
$v\in\mcT_{\mcX}(x^{k+1})$ such that $\|v\|=1$ and $v^T\nabla c_i(x^{k+1})\leq-\delta_c$ for all $i\in\cA(x^{k+1};\theta)$, where $\cA(x^{k+1};\theta)$ is defined in \eqref{def-cAcS}. Let $\bar\cA(x^{k+1};\theta)=\{1,2,\ldots,\tn\}\backslash\cA(x^{k+1};\theta)$. Notice from \eqref{def-cAcS} that $c_i(x^{k+1})<-\theta$ for all $i\in\bar\cA(x^{k+1};\theta)$. In addition, observe from \eqref{rhok-bnd} that $\rho_k\geq\theta^{-1}\Lambda$. 
Using these and $\|\lambda^k_\bfx\|\leq\Lambda$, we obtain that $(\lambda^k_\bfx+\rho_kc(x^{k+1}))_i\leq\Lambda-\rho_k\theta\leq0$ for all $i\in\bar\cA(x^{k+1};\theta)$. By this and the fact that $v^T\nabla c_i(x^{k+1})\leq-\delta_c$ for all $i\in\cA(x^{k+1};\theta)$, one has
\begin{align}
&v^T\nabla c(x^{k+1})\tl^{k+1}_\bfx\overset{\eqref{def-tlx1}}=v^T\nabla c(x^{k+1})[\lambda^k_\bfx+\rho_kc(x^{k+1})]_+=\sum_{i=1}^\tn v^T\nabla c_i(x^{k+1})([\lambda^k_\bfx+\rho_kc(x^{k+1})]_+)_i\nn\\
&=\sum_{i\in\cA(x^{k+1};\theta)}v^T\nabla c_i(x^{k+1})([\lambda^k_\bfx+\rho_kc(x^{k+1})]_+)_i+\sum_{i\in\bar\cA(x^{k+1};\theta)}v^T\nabla c_i(x^{k+1})([\lambda^k_\bfx+\rho_kc(x^{k+1})]_+)_i\nn\\
&\leq-\delta_c \sum_{i\in\cA(x^{k+1};\theta)}([\lambda^k_\bfx+\rho_kc(x^{k+1})]_+)_i=-\delta_c  
\sum_{i=1}^\tn([\lambda^k_\bfx+\rho_kc(x^{k+1})]_+)_i \overset{\eqref{def-tlx1}}= -\delta_c\|\tl^{k+1}_\bfx\|_1. \label{mfcq-ineq}
\end{align}

Since $(x^{k+1},y^{k+1})$ is an $\epsilon_k$-primal-dual stationary point of \eqref{AL-sub}, it follows from  \eqref{AL} and \eqref{stationary} that there exists some $s\in\partial_xF(x^{k+1},y^{k+1})$ such that
\[
\|s+\nabla c(x^{k+1})[\lambda^k_\bfx+\rho_kc(x^{k+1})]_+-\nabla_xd(x^{k+1},y^{k+1})[\lambda^k_\bfy+\rho_kd(x^{k+1},y^{k+1})]_+\|\leq\epsilon_k,
\]
which along with \eqref{def-tlx1} and $\lambda^{k+1}_\bfy=[\lambda^k_\bfy+\rho_xd(x^{k+1},y^{k+1})]_+$ implies that
\beq \label{eps-subgrad}
\|s+\nabla c(x^{k+1})\tl^{k+1}_\bfx -\nabla_xd(x^{k+1},y^{k+1})\lambda^{k+1}_\bfy\|\leq\epsilon_k.
\eeq
In addition, since $v\in\mcT_{\mcX}(x^{k+1})$, there exist $\{z^t\} \subset \mcX$ and $\{\alpha_t\} \downarrow 0$ such that $z^t=x^{k+1}+\alpha_t v+o(\alpha_t)$ for all $t$.  Also, since $s\in\partial_xF(x^{k+1},y^{k+1})$, one has $s=\nabla_x f(x^{k+1},y^{k+1})+s_p$ for some $s_p\in\partial p(x^{k+1})$. Using these and Assumptions \ref{a1} and  \ref{mfcq}(iii), we have
\begin{align}
\langle s, v \rangle &=\langle \nabla_x f(x^{k+1},y^{k+1}), v \rangle + \lim_{t \to \infty} \alpha_t^{-1}\langle s_p, z^t-x^{k+1}\rangle  \nn \\
& = \lim_{t \to \infty} \alpha_t^{-1} (f(z^t,y^{k+1})-f(x^{k+1},y^{k+1})) + \lim_{t \to \infty} \alpha_t^{-1}\langle s_p, z^t-x^{k+1}\rangle \nn \\
& \leq \lim_{t \to \infty} \alpha_t^{-1} (f(z^t,y^{k+1})-f(x^{k+1},y^{k+1})) + \lim_{t \to \infty} \alpha_t^{-1} (p(z^t)-p(x^{k+1}))  \nn \\
& = \lim_{t \to \infty} \alpha_t^{-1} (F(z^t,y^{k+1})-F(x^{k+1},y^{k+1})) \leq L_F \lim_{t \to \infty} \alpha_t^{-1} \|z^t-x^{k+1}\| =L_F, \label{s-bnd}
\end{align}
where the second equality is due to the differentiability of $f$, the first inequality follows from the convexity of $p$ and $s_p\in\partial p(x^{k+1})$, the second inequality is due to the $L_F$-Lipschitz continuity of $F(\cdot, y^{k+1})$, and the last equality follows from $\lim_{t \to \infty} \alpha_t^{-1} \|z^t-x^{k+1}\|=\|v\|=1$.

By \eqref{mfcq-ineq}, \eqref{eps-subgrad}, \eqref{s-bnd},  and $\|v\|=1$, one has
\begin{align*}
\epsilon_k& \geq\|s+\nabla c(x^{k+1})\tl^{k+1}_\bfx -\nabla_xd(x^{k+1},y^{k+1})\lambda^{k+1}_\bfy\|\cdot\|v\|\\
&\geq\langle s+\nabla c(x^{k+1})\tl^{k+1}_\bfx -\nabla_xd(x^{k+1},y^{k+1})\lambda^{k+1}_\bfy,-v\rangle\\
&=-\langle s-\nabla_xd(x^{k+1},y^{k+1})\lambda^{k+1}_\bfy,v\rangle-v^T\nabla c(x^{k+1})\tl^{k+1}_\bfx\\
&\overset{\eqref{mfcq-ineq}}\geq-\langle s, v \rangle - \|\nabla_xd(x^{k+1},y^{k+1})\|\|\lambda^{k+1}_\bfy\| \|v\|+\delta_c\|\tl^{k+1}_\bfx\|_1\\
&\geq -L_F -L_d\|\lambda^{k+1}_\bfy\|+\delta_c\|\tl^{k+1}_\bfx\|_1,
\end{align*}
where the last inequality is due to  \eqref{s-bnd}, $\|v\|=1$ and Assumption \ref{a1}(iii).
Notice from \eqref{rhok-bnd} that \eqref{muk-bnd} holds. It then follows from \eqref{y-cnstr} that $\|\lambda^{k+1}_\bfy\|\leq2\delta_d^{-1}(\Delta+D_\bfy)$, which together with the above inequality and $\epsilon_k\leq 1$ yields
\beq \label{tl-bnd}
\|\tl^{k+1}_\bfx\| \leq \|\tl^{k+1}_\bfx\|_1 \leq \delta_c^{-1}(L_F +L_d\|\lambda^{k+1}_\bfy\|+\epsilon_k)\leq \delta_c^{-1}(L_F +2L_d\delta_d^{-1}(\Delta+D_\bfy)+1).
\eeq
By this and \eqref{def-tlx1}, one can observe that
\[
\|[c(x^{k+1})]_+\| \le \rho^{-1}_k \|[\lambda^k_\bfx+\rho_kc(x^{k+1})]_+\| = \rho^{-1}_k \|\tl^{k+1}_\bfx\| \leq \rho^{-1}_k\delta_c^{-1}(L_F +2L_d\delta_d^{-1}(\Delta+D_\bfy)+1).
\]
Hence, \eqref{x-cnstr2} holds as desired. 

We next show that \eqref{c-complim} holds. Indeed, by $\tl^{k+1}_\bfx\geq 0$, \eqref{x-cnstr2} and \eqref{tl-bnd}, one has 
\begin{align}
\langle\tl^{k+1}_\bfx,c(x^{k+1})\rangle & \ \ \leq \ \langle\tl^{k+1}_\bfx,[c(x^{k+1})]_+\rangle \leq\|\tl^{k+1}_\bfx\|\|[c(x^{k+1})]_+\|\nn \\ 
& \overset{\eqref{x-cnstr2}\eqref{tl-bnd}}\leq\rho^{-1}_k\delta_c^{-2}(L_F +2L_d\delta_d^{-1}(\Delta+D_\bfy)+1)^2. \label{complim-ineq2}
\end{align}
Using a similar argument as for the proof of \eqref{complim-ineq}, we have
\[
-\langle\tl^{k+1}_\bfx,\rho_k^{-1}\lambda^k_\bfx\rangle
\leq \langle\tl^{k+1}_\bfx,c(x^{k+1})\rangle,
\]
which along with $\|\lambda^k_\bfx\|\leq\Lambda$ and \eqref{tl-bnd} yields
\[
\langle\tl^{k+1}_\bfx,c(x^{k+1})\rangle\geq-\rho_k^{-1}\|\tl^{k+1}_\bfx\|\|\lambda^k_\bfx\| \geq -\rho^{-1}_k \delta_c^{-1}(L_F +2L_d\delta_d^{-1}(\Delta+D_\bfy)+1)\Lambda.
\]
The relation \eqref{c-complim} then follows from this and \eqref{complim-ineq2}.
\end{proof}

We are now ready to prove Theorem \ref{complexity} using Lemmas \ref{l-subdcnstr}, \ref{l-subp} and \ref{l-xcnstr2}.

\begin{proof}[\textbf{Proof of Theorem \ref{complexity}}]
(i) Observe from the definition of $K$ in \eqref{K1} and $\epsilon_k=\tau^k$ that $K$ is the smallest nonnegative integer such that $\epsilon_K\leq\varepsilon$. Hence, Algorithm \ref{AL-alg} terminates and outputs $(x^{K+1},y^{K+1})$ after $K+1$ outer iterations. It follows from these and $\rho_k=\epsilon_k^{-1}$ that $\epsilon_K\leq\varepsilon$ and $\rho_K\geq\varepsilon^{-1}$. By this and  \eqref{cond}, one can see that \eqref{muk-1} and \eqref{rhok-bnd} holds for $k=K$. It then follows from Lemmas \ref{l-subdcnstr} and \ref{l-xcnstr2} that \eqref{t1-1}-\eqref{t1-6} hold. 

(ii) Let $K$ and $N$ be given in \eqref{K1} and \eqref{N2}. Recall from Lemma \ref{l-subp} that the number of evaluations of $\nabla f$, $\nabla c$, $\nabla d$, proximal operator of $p$ and $q$ performed by Algorithm \ref{mmax-alg2} at iteration $k$ of Algorithm~\ref{AL-alg} is at most $N_k$, where $N_k$ is given in \eqref{mmax-N}. By this and statement (i) of this theorem, one can observe that the total number of evaluations of $\nabla f$, $\nabla c$, $\nabla d$, proximal operator of $p$ and $q$ performed in Algorithm~\ref{AL-alg} is no more than $\sum_{k=0}^KN_k$, respectively. As a result, to prove statement (ii) of this theorem, it suffices to show that $\sum_{k=0}^KN_k\leq N$. Recall from  \eqref{L-ineq} and Algorithm \ref{AL-alg} that  $\rho_kL_c^2\leq L_k\leq\rho_kL$ and $\rho_k\geq1\geq\epsilon_k$. Using these, \eqref{ho}, \eqref{hM}, \eqref{hT}, \eqref{mmax-omega}, \eqref{mmax-zeta}, \eqref{mmax-Mk} and \eqref{mmax-Tk}, we obtain that
\begin{align}
&1\geq\alpha_k\geq\min\left\{1,\sqrt{4\epsilon_k/(\rho_kD_\bfy L)}\right\}\geq\epsilon_k^{1/2}\rho_k^{-1/2}\alpha, \label{alpha-ineq}\\
&\delta_k\leq(2+\epsilon_k^{-1/2}\rho_k^{1/2}\alpha^{-1})\rho_kL D_\bfx^2+\max\{1/D_\bfy,\rho_kL/4\}D_\bfy^2\leq\epsilon_k^{-1/2}\rho_k^{3/2}\delta, \label{delta-ineq}\\
&M_k\leq\frac{16\max\left\{1/(2\rho_kL_c^2),4/(\epsilon_k^{1/2}\rho_k^{-1/2}\alpha\rho_kL_c^2)\right\}\rho_k}{\left[(3\rho_kL+1/(2D_\bfy))^2/\min\{\rho_kL_c^2,\epsilon_k/(2D_\bfy)\}+3\rho_kL+1/(2D_\bfy)\right]^{-2}\epsilon_k^2}\times\Bigg(\epsilon_k^{-1/2}\rho_k^{3/2}\delta \nn \\
&\ \ \ \ \ \ \ \ +2\epsilon_k^{-1/2}\rho_k^{1/2}\alpha^{-1}\Big(\Delta+\frac{\Lambda^2}{2}+\frac{3}{2}\|\lambda_\bfy^0\|^2+\frac{3(\Delta+D_\bfy)}{1-\tau}+\rho_kd_{\rm hi}^2 +\frac{D_\bfy}{4}+\rho_kL D_\bfx^2\Big)\Bigg) \label{M-ineq} \\
&\leq \frac{16\epsilon_k^{-1/2}\rho_k^{-1/2}\max\left\{1/(2L_c^2),4/(\alpha L_c^2)\right\}\rho_k}{\epsilon_k^2\rho_k^{-4}\left[(3L+1/(2D_\bfy))^2/\min\{L_c^2,1/(2D_\bfy)\}+3L+1/(2D_\bfy)\right]^{-2}\epsilon_k^2}\times(\epsilon_k^{-1/2}\rho_k^{3/2}) \nn \\
&\ \ \times \Bigg(\delta+2\alpha^{-1} \Big(\Delta+\frac{\Lambda^2}{2}+\frac{3}{2}\|\lambda_\bfy^0\|^2+\frac{3(\Delta+D_\bfy)}{1-\tau}+d_{\rm hi}^2+\frac{D_\bfy}{4}+L D_\bfx^2\Big)\Bigg)\leq\epsilon_k^{-5}\rho_k^6M, \nn \\
&T_k\leq\Bigg\lceil16\left(2\Delta+\Lambda+\frac{1}{2}(\tau^{-1}+\|\lambda_\bfy^0\|^2)+\frac{\Delta+D_\bfy}{1-\tau}+\frac{\Lambda^2}{2}+\frac{D_\bfy}{4}\right)\epsilon_k^{-2}\rho_kL\nn \\
&\ \ \ \ \ \ \ \ +8(1+4D_\bfy^2\rho_k^2L^2\epsilon_k^{-2})\rho_k^{-1}-1\Bigg\rceil_+\leq\epsilon_k^{-2}\rho_kT, \nn
\end{align}
where \eqref{M-ineq} follows from \eqref{ho}, \eqref{hM}, \eqref{hT}, \eqref{alpha-ineq}, \eqref{delta-ineq}, $\rho_kL_c^2\leq L_k\leq\rho_kL$, and $\rho_k\geq1\geq\epsilon_k$. 
By the above inequalities, \eqref{mmax-N}, \eqref{L-ineq}, $T\geq1$ and $\rho_k\geq1\geq\epsilon_k$, one has
\begin{align}
&\sum_{k=0}^KN_k\leq \sum_{k=0}^K \left(\left\lceil96\sqrt{2}\left(1+\left(24\rho_kL+4/D_\bfy\right)/(\rho_kL_c^2)\right)\right\rceil+2\right)\max\left\{2,\sqrt{D_\bfy\rho_kL\epsilon_k^{-1}}\right\}\nn\\
&\ \ \ \ \ \ \ \ \ \ \ \ \ \times\left((\epsilon_k^{-2}\rho_kT+1)(\log (\epsilon_k^{-5}\rho_k^6M))_++\epsilon_k^{-2}\rho_kT+1+2\epsilon_k^{-2}\rho_kT\log(\epsilon_k^{-2}\rho_kT+1) \right)\nn\\
&\leq\sum_{k=0}^K\left(\left\lceil96\sqrt{2}\left(1+\left(24L+4/D_\bfy\right)/L_c^2\right)\right\rceil+2\right)\max\left\{2,\sqrt{D_\bfy L}\right\}\epsilon_k^{-1/2}\rho_k^{1/2}\nn\\
&\ \ \ \times\epsilon_k^{-2}\rho_k\left((T+1)(\log (\epsilon_k^{-5}\rho_k^6M))_++T+1+2T\log(\epsilon_k^{-2}\rho_kT+1) \right)\nn\\
&\leq\sum_{k=0}^K\left(\left\lceil96\sqrt{2}\left(1+\left(24L+4/D_\bfy\right)/L_c^2\right)\right\rceil+2\right)\max\left\{2,\sqrt{D_\bfy L}\right\}\nn\\
&\ \ \ \times\epsilon_k^{-5/2}\rho_k^{3/2}T\left(2(\log (\epsilon_k^{-5}\rho_k^6M))_++2+2\log(2\epsilon_k^{-2}\rho_kT) \right)\nn\\
&\leq\sum_{k=0}^K\left(\left\lceil96\sqrt{2}\left(1+\left(24L+4/D_\bfy\right)/L_c^2\right)\right\rceil+2\right)\max\left\{2,\sqrt{D_\bfy L}\right\}T\nn\\
&\ \ \ \times\epsilon_k^{-5/2}\rho_k^{3/2}\left(14\log\rho_k-14\log\epsilon_k+2(\log M)_++2+2\log(2T) \right),\label{sum-N}
\end{align}
By the definition of $K$ in \eqref{K1}, one has $\tau^K\geq\tau\varepsilon$. Also, notice from Algorithm \ref{AL-alg} that $\rho_k=\tau^{-k}$. It then follows from these, \eqref{N2} and \eqref{sum-N} that
\begin{align*}
&\sum_{k=0}^KN_k\leq\sum_{k=0}^K\left(\left\lceil96\sqrt{2}\left(1+\left(24L+4/D_y\right)/L_c^2\right)\right\rceil+2\right)\max\left\{2,\sqrt{D_yL}\right\}T\\
&\ \ \ \ \ \ \ \ \ \ \ \ \ \times\epsilon_k^{-4}\left(28\log(1/\epsilon_k)+2(\log M)_++2+2\log(2T) \right)\\
&= \left(\left\lceil96\sqrt{2}\left(1+\left(24L+4/D_y\right)/L_c^2\right)\right\rceil+2\right)\max\left\{2,\sqrt{D_yL}\right\}T\\
&\ \ \ \times\sum_{k=0}^K\tau^{-4k}\left(28k\log(1/\tau)+2(\log M)_++2+2\log(2T) \right)\\
&\leq \left(\left\lceil96\sqrt{2}\left(1+\left(24L+4/D_y\right)/L_c^2\right)\right\rceil+2\right)\max\left\{2,\sqrt{D_yL}\right\}T\\
&\ \ \ \times\sum_{k=0}^K\tau^{-4k}\left(28K\log(1/\tau)+2(\log M)_++2+2\log(2T) \right)\\
&\leq \left(\left\lceil96\sqrt{2}\left(1+\left(24L+4/D_y\right)/L_c^2\right)\right\rceil+2\right)\max\left\{2,\sqrt{D_yL}\right\}T\\
&\ \ \ \times\tau^{-4K}(1-\tau^4)^{-1}\left(28K\log(1/\tau)+2(\log M)_++2+2\log(2T) \right)\\
&\leq \left(\left\lceil96\sqrt{2}\left(1+\left(24L+4/D_y\right)/L_c^2\right)\right\rceil+2\right)\max\left\{2,\sqrt{D_yL}\right\}T(1-\tau^4)^{-1}\\
&\ \ \ \times \tau^{-4}\varepsilon^{-4}\left(28K\log(1/\tau)+2(\log M)_++2+2\log(2T) \right)\overset{\eqref{N2}}{=} N,
\end{align*}
where the second last inequality is due to $\sum_{k=0}^K\tau^{-4k}\leq \tau^{-4K}/(1-\tau^4)$, and the last inequality is due to $\tau^K\geq\tau\varepsilon$. Hence, statement (ii) of this theorem holds as desired.
\end{proof}

%

\end{document}